\documentclass[11pt]{article}
\usepackage{amssymb}
\usepackage{amsmath}
\usepackage{amsfonts}
\usepackage{graphicx} 
\usepackage{url, hyperref}
\usepackage[margin=1.1in]{geometry}
\usepackage[usenames]{color}
\usepackage[table,xcdraw]{xcolor}
\usepackage{subcaption}

\newtheorem{theorem}{Theorem}
\newtheorem{proposition}[theorem]{Proposition}

\newtheorem{lemma}[theorem]{Lemma}
\newtheorem{definition}[theorem]{Definition}
\newtheorem{hypothesis}[theorem]{Hypothesis} 
\newtheorem{remark}[theorem]{Remark}
\newenvironment{proof}[1][Proof]{\noindent\textbf{#1.} }{\hfill$\blacksquare$\medskip}

\newcommand{\R}{\mathbb{R}}
\newcommand{\T}{\mathbb{T}}

\newcommand{\A}{\mathcal{A}}
\newcommand{\Hp}{H^{2}_{2\pi}}
\newcommand{\Lp}{L^{2}_{2\pi}}
\newcommand{\inner}[2]{\langle #1,#2\rangle}
\newcommand{\norm}[1]{\left\Vert #1\right\Vert}
\newcommand{\abs}[1]{\left\vert #1\right\vert}
\newcommand{\nor}{\mathbf{n}}

\newcommand{\rev}[1]{{#1}} 

\newcommand{\cA}{\mathcal{A}}
\newcommand{\cF}{\mathcal{F}}
\newcommand{\cN}{\mathcal{N}}

\newcommand{\Z}{\mathbb{Z}}

\graphicspath{ {./figures/}}

\title{Platonic constellations of periodic motions in the \\ $(n+1)$-body problem}
\author{Constantineau,~K.\footnote{Department of Mathematics and Statistics, McGill University, 805 Sherbrooke \rev{Street} West, Montreal, QC H3A 0B9, Canada. kevin.constantineau@mail.mcgill.ca},\, Garc\'{\i}a-Azpeitia,~C.\footnote{Departamento de Matem\'aticas y Mec\'anica,
IIMAS-UNAM.
Apdo. Postal 20-126, Col. San \'Angel,
Mexico City, 01000,  Mexico. cgazpe@aries.iimas.unam.mx. ORCID: 0000-0002-6327-1444.} \, and  J.-P.~Lessard\footnote{Department of Mathematics and Statistics, McGill University, 805 Sherbrooke \rev{Street} West, Montreal, QC H3A 0B9, Canada. jp.lessard@mcgill.ca}}
\date{}

\begin{document}
\maketitle

\begin{abstract}
We study the spatial $(n+1)$-body problem formed by one heavy central mass
together with $n$ equal masses placed on a single orbit of a polyhedral rotation
group $H\in\{\T,\mathbb{O},\mathbb{I}\}$, so that $n=\abs{H}\in\{12,24,60\}$.
Imposing the symmetry $q_{L}=L\,q_{I}$ for $L\in H$ reduces the problem to a
single $2\pi$-periodic reference curve, with reduced action
$\A_{H}=\A_{0}+\varepsilon\A_{1}$, in which $\varepsilon$ is the inverse central
mass and $\A_{0}$ is the Kepler action. At $\varepsilon=0$ the critical set
contains, as one connected component, the five-dimensional manifold of Kepler
ellipses of minimal period $2\pi$, which we prove to be a nondegenerate critical
manifold. A Lyapunov--Schmidt reduction along this
manifold turns the continuation problem into the search for nondegenerate critical
points of an explicit function $\Phi(e,\psi)$ of the eccentricity $e$ and the
spatial orientation $\psi$, \rev{a nondegeneracy we verify by a computer-assisted proof.}
We thereby obtain, for each of the three groups, families of periodic solutions of
the $(n+1)$-body problem with $n+1\in\{13,25,61\}$ bodies, bifurcating from Kepler
ellipses and carrying the full tetrahedral, octahedral, or icosahedral symmetry.
\end{abstract}

\section{Introduction}

We consider $n+1$ point masses in $\R^{3}$ interacting through Newtonian
gravitation. One of them is a central body of mass $\mu=\varepsilon^{-1}$ at the
origin, and the other $n$ have unit mass. When the $n$ unit masses are arranged on
a single orbit of a finite rotation group $H<SO(3)$ and move so that the whole
configuration stays $H$-symmetric for all time, the $(n+1)$-body problem reduces
to a low-dimensional variational problem for a single representative body. In the
planar case, with $H$ cyclic, this is the classical Maxwell ring studied in~\cite{GaIz13}, whose restricted problem carries symmetry-breaking bifurcations followed numerically in~\cite{CDG16}. Here we treat the three genuinely
nonplanar cases, in which $H$ is the tetrahedral, octahedral, or icosahedral
rotation group, so that $n=\abs{H}\in\{12,24,60\}$. The same three
groups organize the $N$-vortex problem on the sphere, where prescribing the
symmetry and reducing by it produces families of small nonlinear oscillations
emanating from the platonic equilibria~\cite{GaGN22}.

\paragraph{The solutions.}
For every sufficiently small $\varepsilon>0$, that is, for a large enough central
mass, we construct motions of the following form. The central body stays at the
origin, and the $n$ unit masses are the $H$-orbit of a single reference curve,
\begin{equation}\label{eq:solform}
q_{0}(t)=0,\qquad
q_{L}(t)=L\,u_{\varepsilon}\bigl(\varepsilon^{-1/2}t\bigr)\quad(L\in H),
\end{equation}
one body per group element, of period $2\pi\varepsilon^{1/2}$. The reference curve
$u_{\varepsilon}$ is a perturbed Kepler ellipse, $2\pi$-periodic in the rescaled
time $\tau=\varepsilon^{-1/2}t$. Figure~\ref{fig:intro-oct2} shows one of them, an
octahedral solution with $24+1=25$ bodies, whose eccentricity tends as
$\varepsilon\to0$ to a limit enclosed in Table~\ref{tab:cap-ecc}\rev{, close to}
$\bar e=0.962477$. The other solutions we obtain are
drawn in Figures~\ref{fig:cap-grid1} and~\ref{fig:cap-grid2} of
Section~\ref{ss:cap-results}. In all of these figures the $n$ unit masses lie at
every instant on a single $H$-orbit, so they are the vertices of one polyhedron.

Imposing the symmetry \eqref{eq:solform} on the action of the
$(n+1)$-body problem, and rescaling time so that the period becomes $2\pi$,
reduces that action to
\begin{equation}\label{eq:intro-AH}
\A_{H}(u;\varepsilon)=\A_{0}(u)+\varepsilon\,\A_{1}(u),
\end{equation}
written explicitly in \eqref{eq:AH}, in which $\varepsilon=\mu^{-1}$ is the inverse
central mass, $\A_{0}$ is the Kepler action of the reference body, and $\A_{1}$ is
the mutual interaction of the $n$ unit masses.
The periodic solution $u_{\varepsilon}$ of \eqref{eq:solform} is a critical point
of $\A_{H}(\,\cdot\,;\varepsilon)$ on $2\pi$-periodic curves, and conversely every
such critical point yields a motion of the form \eqref{eq:solform}.

The existence of solutions of this form, isolated up to time
translation, is a consequence of the symmetry
breaking produced by the perturbation $\varepsilon\A_{1}$. The Kepler action
$\A_{0}$ is invariant under the full orthogonal group $O(3)$, so at $\varepsilon=0$ no orientation is distinguished. If a Kepler ellipse is a solution, then so is every rotation of it, and the critical set of $\A_{0}$ therefore contains, around each ellipse, a whole continuum of rotated solutions. The interaction term $\A_{1}$ is invariant only under the rigid motions that normalize $H$, that is, under the finite normalizer $N_{O(3)}(H)$. Switching on $\varepsilon$ thus breaks $O(3)$ down to the normalizer, the continuum of orientations is reduced to isolated ones, and those come in orbits of the normalizer, which is why the search \rev{in} Section~\ref{sec:cap} may be confined to a fundamental domain of that normalizer.

\begin{figure}[h!]
\centering
\includegraphics[width=0.8\textwidth]{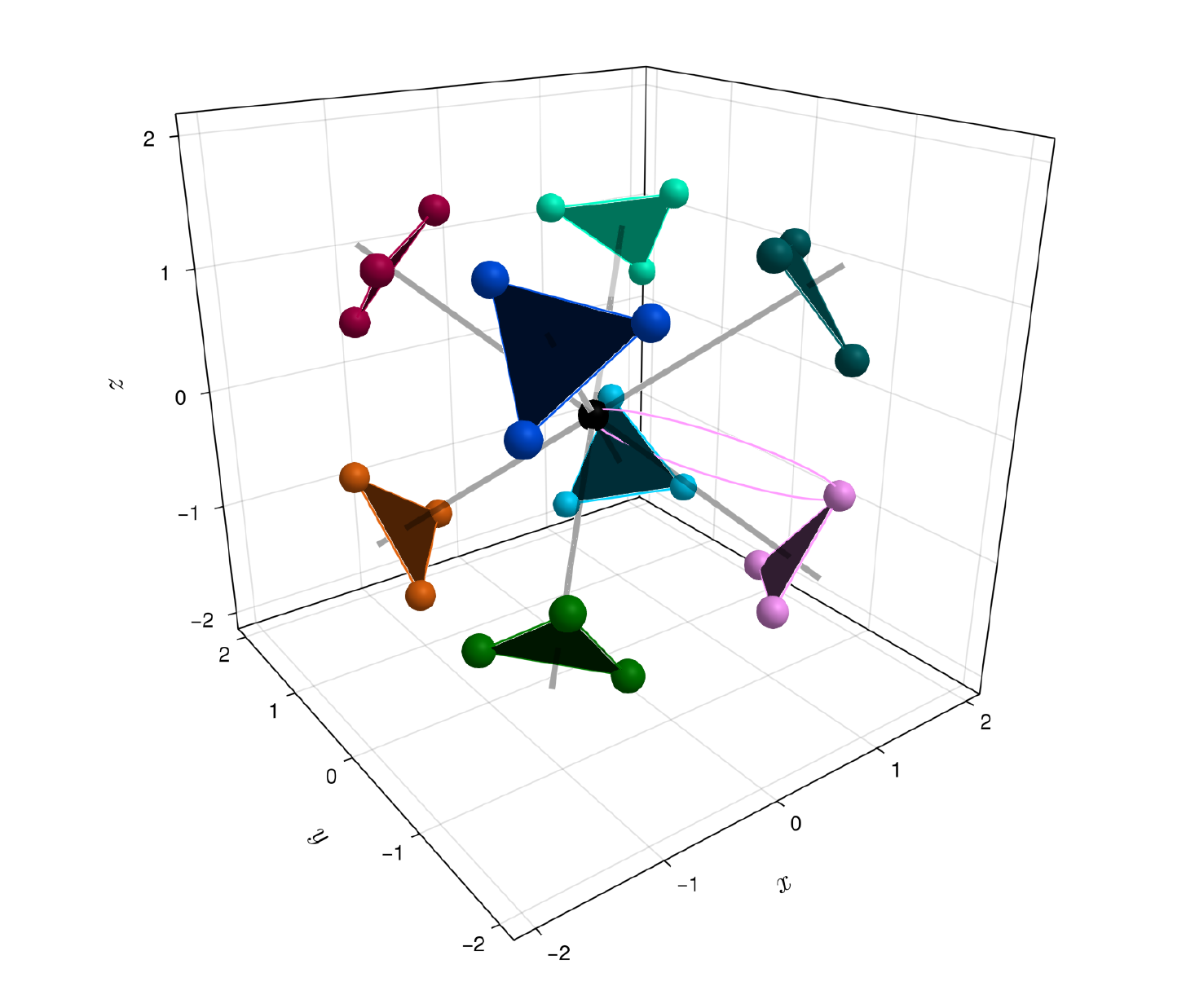}
\caption{An octahedral solution with $n=24$ unit masses and one central mass, in the
limit $\varepsilon\to0$, where it is the Kepler ellipse of eccentricity
$\bar e=0.962477$ and its $H$-orbit. The drawing conventions are
those of Figures~\ref{fig:cap-grid1} and~\ref{fig:cap-grid2}. The central mass is
the black ball at the origin, and the $n$ unit masses are one $H$-orbit, drawn as
the vertices of a polyhedron whose lit faces give the perspective and display the
symmetry. The reference body is the pink ball and its ellipse the pink curve, and
the grey rods are rotation axes of $H$.}
\label{fig:intro-oct2}
\end{figure}

\paragraph{The method.}
Our method is continuation from $\varepsilon=0$, by a Lyapunov--Schmidt reduction
along the critical set of the Kepler action $\A_{0}$, whose critical points are the
elliptic solutions of the Kepler problem. Imposing
the period $2\pi$ restricts that set to the manifold $M$ of Kepler ellipses of that
minimal period, which has dimension five, the parameters being the eccentricity $e$, the
three angles $\psi=(\psi_{1},\psi_{2},\psi_{3})$ that place the ellipse in space,
and the phase along the orbit. We prove that $M$ is a \emph{nondegenerate critical manifold}
(Proposition~\ref{thm:kepler}), meaning that at every $u\in M$ the kernel of the
second variation $\nabla^{2}\A_{0}(u)$ equals the tangent space $T_{u}M$,
the nondegeneracy condition of Bott~\cite{Bott54}.
The phase is a symmetry of the problem rather than a genuine
parameter. \rev{The reduction is carried out along the whole kernel of the Hessian at the base ellipse, the phase direction included,}
and the phase then leaves the problem on its own, because $\A_{H}$ is invariant
under time translation and the reduced function therefore depends on $e$ and
$\psi$ alone. What has to be solved is a system of four equations in the four
unknowns $(e,\psi)$, and these are the parameters we use throughout. They are
explicit coordinates, since a point of $M$ is, at a fixed phase, the curve
\begin{equation}\label{eq:intro-chart}
u_{e,\psi}(\tau)=R(\psi)\,u_{e}(\tau),\qquad
R(\psi)=e^{\psi_{1}J_{1}}\,e^{\psi_{2}J_{2}}\,e^{\psi_{3}J_{3}}\in SO(3),
\end{equation}
where $u_{e}$ is the $2\pi$-periodic Kepler ellipse of eccentricity $e$ in
the reference plane $\{z=0\}$ and $J_{1},J_{2},J_{3}$ are the infinitesimal
generators of the rotations about the coordinate axes, \rev{in the sign convention fixed below}. The
two normal angles $\psi_{1},\psi_{2}$ parametrize the unit
normal of the orbit plane on the sphere, and the in-plane angle $\psi_{3}$ is the angle
of the ellipse within that plane, its argument of pericentre.

The reduction produces an explicit function $\Phi(e,\psi)$, the restriction to $M$ of the
interaction term $\A_{1}$, and each nondegenerate critical point of
$\Phi$ yields a branch of symmetric periodic solutions of the $(n+1)$-body
problem for all small $\varepsilon>0$, with the asymptotic expansion
\begin{equation}\label{eq:uexpand}
u_{\varepsilon}=R(\psi_{\varepsilon})\,u_{e_{\varepsilon}}
+O(\varepsilon),
\qquad
(e_{\varepsilon},\psi_{\varepsilon})=(e_{0},\psi_{0})+O(\varepsilon),
\end{equation}
where $(e_{0},\psi_{0})$ is a nondegenerate critical point of $\Phi$ from which the branch is
continued. To leading order, then, the motion is $n$ copies of one
Kepler ellipse rigidly permuted by $H$.
Each solution comes as an $S^{1}$ orbit of time translations,
$\{u_{\varepsilon}(\,\cdot+\varphi):\varphi\in S^{1}\}$.
This is the content of the two abstract
theorems of the paper, which differ only in the chart used.
The chart \eqref{eq:intro-chart} degenerates at the circle, and
also at the pole $\cos\psi_{2}=0$ of the Euler angles, which none of our critical
points approaches. Theorem~\ref{thm:cont} of
Section~\ref{sec:cont} treats a base ellipse of positive eccentricity, and
Theorem~\ref{thm:circular} of Section~\ref{sec:circular} treats the circular case
$e_{0}=0$. Producing periodic
motions of the $(n+1)$-body problem, for small $\varepsilon$, is thereby made
equivalent to producing
nondegenerate critical points of one explicit function of four variables.

\paragraph{Finding the critical points and proving that they exist.}
It remains to find those critical points. Each of them is first located
numerically, and its existence and nondegeneracy are then proven in
Section~\ref{sec:cap} by a computer-assisted proof. The gradient and Hessian of
$\Phi$, written in closed form in Sections~\ref{sec:reduced}
and~\ref{sec:circular}, are evaluated with interval arithmetic. A
Newton--Kantorovich theorem in its radii-polynomial
form then turns a numerical approximation of a critical point into a proof that a true
one lies within an explicit radius of it, the nondegeneracy being proven on the
same enclosure. The same scheme \rev{was used to prove} the existence of spatial relative
equilibria, and of the normal frequencies from which their periodic solutions
bifurcate, in the Coulomb $(n+1)$-body problem of~\cite{CoGaLe2021}, in which a
fixed nucleus of charge $\mu$ replaces the central mass.
Carried out at each of the points listed in
Tables~\ref{tab:cap-e0} and~\ref{tab:cap-ecc}, this yields eleven critical points,
two for $\T$, four for $\mathbb{O}$, and five for $\mathbb{I}$, and hence, through
the two abstract theorems, eleven families of periodic solutions.

Those two tables, in Section~\ref{ss:cap-results}, are the
concrete output of the paper. Angles are in radians reduced to $[0,2\pi)$, and a
bar denotes a numerical approximation, so that $\bar e$ and $\bar\psi_{j}$
approximate the coordinates $e_{0}$ and $\psi_{0,j}$ of the critical point the row
encloses. Table~\ref{tab:cap-e0} lists the critical orientations of the circular
orbit, one for $\T$, two for $\mathbb{O}$ and four for $\mathbb{I}$. The in-plane
angle $\psi_{3}$ is absent there, since at the circle it only shifts the true
anomaly. The eccentricity gradient vanishes identically by
\eqref{eq:gradeps}, so only the two orientation equations remain.
Table~\ref{tab:cap-ecc} lists the critical points of positive eccentricity, one
for $\T$, two for $\mathbb{O}$ and one for $\mathbb{I}$. Neither table is claimed
to be exhaustive, as Remark~\ref{rem:notexhaustive} records.

\paragraph{Relation to the global methods.}
The continuation is purely \emph{local}, in the sense that each nondegenerate
critical point of $\Phi$ produces a single branch emanating from one Kepler
ellipse, and only for $\varepsilon$ small. It is natural to ask
for a \emph{global} principle instead, at a fixed $\varepsilon>0$
and with the full reduced action, of which
$\varepsilon\Phi$ is the leading term, on the whole family of Kepler ellipses
$e\in[0,1)$. Minimizing that action, or producing critical points of it by a
topological argument, would give solutions with no nondegeneracy hypothesis at
all. Both routes have been taken for problems close to this one.
Neither is available here, and the reason lies at the two ends of the family. On one
side the interaction along the ellipse family is defined only on the orientations
whose orbit plane avoids the
rotation axes of $H$, and it blows up as the collisions between
the unit masses are approached. On the
other, at the boundary $e\to1$, where the ellipse degenerates onto
a collision with the central mass, it does \emph{not} blow up, having a
finite limit there, so it offers no coercivity barrier and nothing
prevents the critical value produced by a global argument from being
attained in that limit. That limit is a genuine point of the
closure of the family, since at a fixed period the ellipses lie on one level set
of negative energy of the Kepler problem, which the regularization of
Moser~\cite{Moser70} compactifies by adjoining exactly those degenerate orbits.

The closest reference is the variational one, Fenucci and
Gronchi~\cite{FeGr23}. They treat the same configuration, one body of large mass
together with $n$ unit masses forming a single orbit of a finite rotation group,
the Platonic ones included, and they minimize the action under symmetry and
topological constraints. In that approach the collisions are excluded rather than
regularized. The estimate of Marchal~\cite{Mar02} shows that a minimizer of the
action between fixed endpoints has no double collision at an interior time, and
Ferrario and Terracini~\cite{FeTe04} give a condition on the symmetry group under
which the action restricted to the equivariant loops is coercive and its
minimizers are free of collisions, with no strong force assumption. The minimizers
of~\cite{FeGr23} are likewise proved to be free of collisions, along a diverging
sequence of values of the central mass. Their asymptotic statement is that the
action $\Gamma$-converges, as the central mass grows, to the action of a
Kepler problem. That convergence does not determine which ellipse is selected, the
Kepler action being constant on $M$ by Step 3 of the proof of
Theorem~\ref{thm:cont}. In
the polyhedral case their symmetry constraint is a relation
$u_{I}(t+T/m)=R\,u_{I}(t)$, with $R$ in the rotation group and $m$ an integer. For
$m\ge2$ it forces the radius $\abs{u_{I}}$ to have period $T/m$, and among
the Kepler ellipses of minimal period $T$ only the circles do so. The four
families of Table~\ref{tab:cap-ecc} are, to leading order in $\varepsilon$,
ellipses of positive eccentricity, so they lie outside the classes over which that
minimization is taken. We obtain them from the critical points of $\Phi$, at the order at which
the ellipses are first distinguished from one another.

The second closest reference is the topological one, Boscaggin,
Ortega and Zhao~\cite{BOZ19}. For a Kepler problem in dimension two or three with
a small time-periodic perturbation they produce periodic solutions of the period
of the perturbation by a Lusternik--Schnirelmann argument. Such an argument bounds
the number of solutions below by the category of the space on which it is run, and
the compactness it needs is supplied by the regularization. What it does not
supply is a way to keep the solutions off the collision set that the
regularization adds, and they are accordingly obtained in a generalized sense that
admits collisions. In the plane they are identified with the solutions given by
the Levi-Civita regularization, whereas in space the corresponding statement for
the Kustaanheimo--Stiefel variables is left open there. We
obtain instead solutions of the unregularized problem, with $e<1$ and free of collisions,
by the local implicit function theorem and computer-assisted proofs.

The construction continues from $\varepsilon=0$, that is from an infinite central
mass, so the solutions we obtain exist for $\mu$ large. At the opposite end,
$\mu=0$, the central body is absent and one is in the classical $n$-body problem,
where periodic solutions with these same polyhedral symmetries were obtained by
Fusco, Gronchi, and Negrini~\cite{FGN11}, and for the potentials
$1/r^{\alpha}$ with $\alpha>1$ by Fusco and Gronchi~\cite{FuGr14}. The stability of
these motions is studied by Fenucci and Gronchi~\cite{FeGr18}, where the
instability of some of them is proved by interval arithmetic. Those solutions and
ours sit at opposite ends of the range of $\mu$ and are reached by different
means. Whether a global branch in $\mu$ joins
the two families is beyond a perturbative construction of this
kind, and it is the question we plan to take up in a sequel to this paper.

\paragraph{Plan of the paper.}
Section~\ref{sec:reduction} reduces the $(n+1)$-body problem, by the symmetry
\eqref{eq:symm} and a rescaling of time, to the action
$\A_{H}=\A_{0}+\varepsilon\A_{1}$ of a single $2\pi$-periodic curve.
Sections~\ref{sec:eccentric} and~\ref{sec:circular} continue the Kepler ellipses of
that action, the first for positive eccentricity, the second at the circle, in the
eccentricity-vector chart that the circle requires.
Section~\ref{sec:cap} then verifies, by
computer-assisted proof, the nondegeneracy both continuations require, and reports
the enclosures in Tables~\ref{tab:cap-e0} and~\ref{tab:cap-ecc}.

\section{The symmetric reduction}\label{sec:reduction}

We first impose the symmetry, which reduces the action of the $n+1$ bodies to a
single integral over one curve. A rescaling of time then puts that action in the
perturbative form $\A_{0}+\varepsilon\A_{1}$, from which the continuation starts.
The three groups are fixed last, by explicit generators, which determines the sum
over $H$ in the interaction term.

\subsection{Reduction of the action by the symmetry}\label{ss:red-symm}

The action of the $n+1$ bodies is
\[
\A(q)=\int \mu\frac{\abs{\dot q_{0}}^{2}}{2}
+\sum_{j=1}^{n}\frac{\abs{\dot q_{j}}^{2}}{2}-V(q)\,dt,\qquad
q=(q_{0},\dots,q_{n}),
\]
with gravitational potential
\[
V(q)=-\mu\sum_{j=1}^{n}\frac{1}{\abs{q_{j}-q_{0}}}
-\sum_{1\le j<k\le n}\frac{1}{\abs{q_{j}-q_{k}}} .
\]
It is invariant under the left action of $O(3)\times S_{n}$, in which $O(3)$
rotates and reflects all bodies simultaneously and $S_{n}$ permutes the $n$ unit
masses,
\[
(R,\sigma)\cdot q=(Rq_{0},Rq_{\sigma^{-1}(1)},\dots,Rq_{\sigma^{-1}(n)}).
\]

Let $H<SO(3)$ be a finite group of order $n$, and number its elements
$L_{1}=I,\dots,L_{n}$, in which $I$ is the identity matrix, and hence the identity
of $H$. We identify the $j$-th body with the corresponding group
element by writing $q_{j}=q_{L_{j}}$, and we let
$\widetilde H=\{(L,L):L\in H\}<O(3)\times S_{n}$ be the associated diagonal
subgroup. The motions considered throughout are those whose configuration lies, at
every instant, in the fixed-point space of $\widetilde H$, that is, those
satisfying
\begin{equation}\label{eq:symm}
q_{0}=0,\qquad q_{L}=L\,q_{I}\quad(L\in H),
\end{equation}
where $q_{I}$ denotes the position of the reference body $L_{1}=I$. Thus the
entire configuration is the single $H$-orbit of $q_{I}$, and the central body sits
at the origin.

\begin{lemma}\label{lem:central}
Suppose $H$ acts on $\R^{3}$ with no nonzero fixed vector, that is,
$\mathrm{Fix}(H):=\{x\in\R^{3}:Lx=x\ \forall L\in H\}=\{0\}$, equivalently the
trivial representation does not occur in $\R^{3}$. Then $\sum_{L\in H}L=0$, the
configurations \eqref{eq:symm} have their centre of mass at the origin, and
$q_{0}=0$ is consistent with the equations of motion, so the central body remains
fixed. The hypothesis holds for $\T,\mathbb{O},\mathbb{I}$, whose standard
three-dimensional representation is irreducible and nontrivial.
\end{lemma}

\begin{proof}
Consider the averaging operator
\[
\Pi:=\frac{1}{\abs H}\sum_{L\in H}L,
\]
the arithmetic mean of the group elements acting on $\R^{3}$. We claim $\Pi$ is the
orthogonal projection of $\R^{3}$ onto $\mathrm{Fix}(H)$. Indeed, for every
$M\in H$ one has $M\Pi=\frac1{\abs H}\sum_{L}ML=\Pi$, because $L\mapsto ML$
permutes $H$. Hence every vector in the range of $\Pi$ is fixed by all of $H$, so
$\operatorname{ran}\Pi\subseteq\mathrm{Fix}(H)$. Conversely $\Pi x=x$ whenever
$x\in\mathrm{Fix}(H)$. Together these give $\Pi^{2}=\Pi$, and self-adjointness
follows from $L^{\top}=L^{-1}\in H$. Thus $\Pi$ is the orthogonal projection onto
$\mathrm{Fix}(H)$. By hypothesis $\mathrm{Fix}(H)=\{0\}$, so $\Pi=0$ and therefore
$\sum_{L\in H}L=0$. Consequently
$\sum_{L\in H}q_{L}=\bigl(\sum_{L}L\bigr)q_{I}=0$, so the $n$ unit masses have
vanishing barycentre. Finally, the gravitational force exerted on the central body
is proportional to $\sum_{L}q_{L}/\abs{q_{L}}^{3}$. Since
$\abs{q_{L}}=\abs{q_{I}}$ is independent of $L$ and $\sum_{L}q_{L}=0$, this force
vanishes, so $q_{0}=0$ solves its equation of motion.
\end{proof}

\begin{remark}
The condition $\mathrm{Fix}(H)=\{0\}$ is sharp, not merely sufficient. Since $\Pi$
is the projection onto $\mathrm{Fix}(H)$, the identity $\sum_{L}L=0$ holds
\emph{if and only if} $H$ fixes no nonzero vector. It fails for any group \rev{whose elements share a} common rotation axis. For instance, in the planar problem the cyclic group $C_{n}$
acts on the orbit plane $\R^{2}$, where $\mathrm{Fix}=\{0\}$ and
$\sum_{L\in C_{n}}L=0$ for $n\ge2$. Embedded in $\R^{3}$ as the
rotations about the $z$-axis, that same $C_{n}$ fixes that axis, and
$\sum_{L\in C_{n}}L$ is $n$ times the orthogonal projection onto it.
\end{remark}

Substituting \eqref{eq:symm} and using $\abs{Lq_{I}}=\abs{q_{I}}$ and
$\abs{\dot q_{L}}=\abs{\dot q_{I}}$ for $L\in H<SO(3)$, the action reduces to a
constant multiple of a single integral over the reference body $u:=q_{I}$,
\[
\A(q)=n\int \frac{\abs{\dot u}^{2}}{2}
+\frac{\mu}{\abs{u}}
+\frac12\sum_{L\in H\setminus\{I\}}\frac{1}{\abs{u-Lu}}\,dt .
\]
The overall factor $n$ and the (vanishing) kinetic term of the central body do not
affect the Euler--Lagrange equations and are dropped from here on. The group enters
this integral only through the finite sum over $H\setminus\{I\}$, which is made
explicit in Section~\ref{sec:groups}.

\subsection{The rescaled reduced action}\label{ss:red-scaled}

We now rescale time, so as to make the limit of a large central
mass a regular one. After the rescaling the mass enters only through
$\varepsilon=\mu^{-1}$, and the reduced action extends smoothly to
$\varepsilon=0$, where it is the Kepler action. Set
\begin{equation}\label{eq:scale}
u(\tau)=q_{I}(t),\qquad \tau=\mu^{1/2}t,\qquad \varepsilon=\mu^{-1},
\end{equation}
so that a $2\pi$-periodic curve $u(\tau)$ corresponds to a
$2\pi\varepsilon^{1/2}$-periodic motion of the original problem. Substituting
\eqref{eq:scale} and discarding the resulting constant factor $\mu^{1/2}$, the
reduced action on $2\pi$-periodic curves $u$ becomes
\begin{equation}\label{eq:AH}
\A_{H}(u;\varepsilon)=\int_{0}^{2\pi}
\frac{\abs{\dot u}^{2}}{2}+\frac{1}{\abs{u}}
+\frac{\varepsilon}{2}\!\!\sum_{L\in H\setminus\{I\}}\!\!
\frac{1}{\abs{u-Lu}}\,d\tau
=\A_{0}(u)+\varepsilon\,\A_{1}(u),
\end{equation}
where $\A_{0}$ is the Kepler action and $\A_{1}$ the interaction term,
\[
\A_{0}(u)=\int_{0}^{2\pi}\frac{\abs{\dot u}^{2}}{2}+\frac{1}{\abs{u}}\,d\tau,
\qquad
\A_{1}(u)=\frac12\int_{0}^{2\pi}\sum_{L\ne I}\frac{1}{\abs{u-Lu}}\,d\tau .
\]
Here $\varepsilon=\mu^{-1}$ is the small parameter, the inverse central mass, and
the symmetric periodic solutions are exactly the critical points of
$\A_{H}(\,\cdot\,;\varepsilon)$ on the space of $2\pi$-periodic curves.

\paragraph{Euler--Lagrange equations.}
The gradient of the Kepler action, taken in the $\Lp$ inner product, is
\[
\nabla\A_{0}(u)=-\ddot u-\frac{u}{\abs{u}^{3}} .
\]
For the interaction term, recall that $L^{-1}$ belongs to $H$
whenever $L$ does, and that
$L^{-1}=L^{\top}$. Differentiating
$\abs{u-Lu}^{-1}=\bigl((I-L)u\cdot(I-L)u\bigr)^{-1/2}$ and using
$(I-L)^{\top}(I-L)=2I-L-L^{\top}$, the term indexed by $L$ contributes
$-(2I-L-L^{\top})u/\abs{(I-L)u}^{3}$. Pairing $L$ with $L^{-1}=L^{\top}$, which
satisfies $\abs{(I-L^{\top})u}=\abs{(I-L)u}$ since
$\abs{u-L^{\top}u}=\abs{Lu-u}$, and summing over the group gives
\begin{equation}\label{eq:EL1}
\nabla\A_{1}(u)=-\frac12\sum_{L\ne I}
\frac{(2I-L-L^{\top})u}{\abs{(I-L)u}^{3}}
=-\sum_{L\ne I}\frac{u-Lu}{\abs{u-Lu}^{3}} .
\end{equation}
The Euler--Lagrange equation of \eqref{eq:AH} is therefore
$\nabla\A_{0}(u)+\varepsilon\,\nabla\A_{1}(u)=0$.

\paragraph{Symmetries.}
The Kepler action $\A_{0}$ is invariant under time translation
$u(\tau)\mapsto u(\tau+\varphi)$ and under the full orthogonal group, since
$\A_{0}(Ru)=\A_{0}(u)$ for all $R\in O(3)$. The interaction term $\A_{1}$ is
invariant under time translation and under the normalizer $N_{O(3)}(H)$, because
$\abs{Ru-LRu}=\abs{u-(R^{-1}LR)u}$ and $R^{-1}LR$ runs over $H$ precisely when
$R\in N_{O(3)}(H)$.

Throughout we use the \rev{basis} of $\mathfrak{so}(3)$,
\[
J_{1}=\begin{pmatrix}0&0&0\\0&0&-1\\0&1&0\end{pmatrix},\quad
J_{2}=\begin{pmatrix}0&0&-1\\0&0&0\\1&0&0\end{pmatrix},\quad
J_{3}=\begin{pmatrix}0&-1&0\\1&0&0\\0&0&0\end{pmatrix},
\]
in which $J_{3}$ generates the rotation about the $z$-axis, the
axis normal to the reference orbit plane $\{z=0\}$ used below. 
\subsection{The three polyhedral groups}\label{sec:groups}

It remains to fix the three groups, the only data still implicit in \eqref{eq:AH}.
In each case $H<SO(3)$ is the rotation group of a regular polyhedron, acting by its
standard three-dimensional representation. That representation is irreducible and
nontrivial, so $\mathrm{Fix}(H)=\{0\}$ and Lemma~\ref{lem:central} applies, giving
$\sum_{L\in H}L=0$. In $SO(3)$ the normalizer of $\T$ is
$\mathbb{O}$, while $\mathbb{O}$ and $\mathbb{I}$ are their own normalizers. For
each group we record a pair of generators $A,B$ with their
defining relations. The group is the closure of $\{A,B\}$ under multiplication, and
this closure is how the $\abs H$ elements of $H$ are enumerated whenever the sum
over $H$ has to be evaluated.

\paragraph{Tetrahedral group, $n=12$.}
The generators
\[
A=\begin{pmatrix}1&0&0\\0&-1&0\\0&0&-1\end{pmatrix},\qquad
B=\begin{pmatrix}0&1&0\\0&0&1\\1&0&0\end{pmatrix}
\]
satisfy $A^{2}=B^{3}=(AB)^{3}=I$ and generate
\[
\T=\langle A,B\mid A^{2}=B^{3}=(AB)^{3}=I\rangle\cong A_{4},
\qquad \abs{\T}=12 .
\]
Here $A$ is the half-turn about the $x$-axis and $B$ the order-three rotation about
the body diagonal $(1,1,1)$.
Geometrically, $\T$ is the rotation group of the regular tetrahedron with vertices
$(1,1,1)$, $(1,-1,-1)$, $(-1,1,-1)$, and $(-1,-1,1)$. The
twelve elements split by rotation type as $1$ identity, $3$ half-turns about the
three coordinate ($2$-fold) axes, and $8$ third-turns about the four
body-diagonal ($3$-fold) axes $(\pm1,\pm1,\pm1)$.

\paragraph{Octahedral group, $n=24$.}
The generators
\[
A=\begin{pmatrix}0&-1&0\\1&0&0\\0&0&1\end{pmatrix},\qquad
B=\begin{pmatrix}-1&0&0\\0&0&1\\0&1&0\end{pmatrix}
\]
satisfy $A^{4}=B^{2}=(AB)^{3}=I$ and generate
\[
\mathbb{O}=\langle A,B\mid A^{4}=B^{2}=(AB)^{3}=I\rangle\cong S_{4},
\qquad \abs{\mathbb{O}}=24 .
\]
Concretely, $\mathbb{O}$ is the group of the $24$ signed permutation matrices of
determinant $+1$. Here $A$ is the quarter-turn about the $z$-axis and $B$ the
half-turn about the axis $(0,1,1)$. Geometrically, $\mathbb{O}$ is the rotation
group of the cube, equivalently of the regular octahedron with vertices
$\pm e_{1},\pm e_{2},\pm e_{3}$, the coordinate axes joining opposite vertices of
that octahedron and the body diagonals joining opposite vertices of the cube. The
elements split by rotation type as $1$
identity, $6$ quarter-turns and $3$ half-turns about the three coordinate
($4$-fold) axes, $8$ third-turns about the four body-diagonal ($3$-fold) axes, and
$6$ half-turns about the six edge ($2$-fold) axes.

\paragraph{Icosahedral group, $n=60$.}
Let $\phi=\tfrac12(1+\sqrt5)$ be the golden ratio, so that $\phi-1=\phi^{-1}$. We
\rev{take} for $B$ an order-three rotation about the body diagonal
$(1,1,1)$ and take for $A$ the half-turn about the axis $(\phi,1,\phi-1)$, one of
the two-fold axes of $\mathbb{I}$,
\[
A=\frac12\begin{pmatrix}\phi-1&\phi&1\\[2pt]\phi&-1&\phi-1\\[2pt]
1&\phi-1&-\phi\end{pmatrix},\qquad
B=\begin{pmatrix}0&0&1\\1&0&0\\0&1&0\end{pmatrix} .
\]
They satisfy $A^{2}=B^{3}=(AB)^{5}=I$ and generate
\[
\mathbb{I}=\langle A,B\mid A^{2}=B^{3}=(AB)^{5}=I\rangle\cong A_{5},
\qquad \abs{\mathbb{I}}=60 .
\]
Geometrically, $\mathbb{I}$ is the rotation group of the regular icosahedron,
equivalently of the dodecahedron. The sixty
elements split by rotation type as $1$ identity, $24$ fifth-turns about the six
$5$-fold (vertex) axes, $20$ third-turns about the ten $3$-fold (face) axes, and
$15$ half-turns about the fifteen $2$-fold (edge) axes.

\section{The eccentric case, $0<e_{0}<1$}\label{sec:eccentric}

Here the base eccentricity is strictly positive, $0<e_{0}<1$, the circular case
being deferred to Section~\ref{sec:circular}. We first show that the manifold $M$ of
$2\pi$-periodic Kepler ellipses is a nondegenerate critical manifold of the Kepler
action, which is what a Lyapunov--Schmidt reduction along $M$ needs. That reduction
is Theorem~\ref{thm:cont}, and it replaces the continuation problem by the search
for nondegenerate critical points of the reduced function $\Phi=\A_{1}|_{M}$, which
the last subsection writes in closed form together with its first two derivatives.

\subsection{The Kepler manifold and its nondegeneracy}\label{sec:kepler}

At $\varepsilon=0$ the critical points of $\A_{0}$ on $2\pi$-periodic curves are
exactly the $2\pi$-periodic Kepler ellipses, that is, for each integer
$k\ge1$, the ellipses of minimal period $2\pi/k$ traversed $k$ times. Call that
family $M_{k}$. The $M_{k}$ are the connected components of the
critical set, being separated by the semi-major axis, as recorded after
\eqref{eq:cnorm} below. We take $M=M_{1}$ and leave the
subharmonic components $k\ge2$ aside. We first fix a planar reference
ellipse of eccentricity $e\in[0,1)$ carrying the correct mean motion. In polar
coordinates $(r,\theta)$ in the plane $\{z=0\}$ it is
\begin{equation}\label{eq:ellipse}
u_{e}=r(\theta)\,(\cos\theta,\sin\theta,0),\qquad
r=\frac{c_{e}^{2}}{1+e\cos\theta},\qquad
\dot\theta=\frac{c_{e}}{r^{2}}=\frac{(1+e\cos\theta)^{2}}{c_{e}^{3}},
\end{equation}
where $r$ is the distance to the central mass, $\theta$ is the true anomaly
measured from pericentre, and $c_{e}>0$ is the angular momentum (the constant
$r^{2}\dot\theta$ of Kepler's second law).

The angular momentum $c_{e}$ is fixed
by the eccentricity through the requirement that the period be exactly $2\pi$
(cf.~\cite[(7.17)]{MeHa91}),
\[
2\pi=\int_{0}^{2\pi}d\tau=c_{e}^{3}\int_{0}^{2\pi}(1+e\cos\theta)^{-2}\,d\theta .
\]
We evaluate this normalizing integral once and for all. Write
$\gamma_{m}=\int_{0}^{2\pi}(1+e\cos\theta)^{-m}\,d\theta$ for the relevant moments.
The standard value $\gamma_{1}=2\pi(1-e^{2})^{-1/2}$, obtained from the
substitution $t=\tan\tfrac\theta2$, together with the elementary identity
$\cos\theta\,(1+e\cos\theta)^{-2}
=e^{-1}\bigl[(1+e\cos\theta)^{-1}-(1+e\cos\theta)^{-2}\bigr]$, gives, on
differentiating $\gamma_{1}$ in $e$,
\[
\gamma_{1}'=-\!\int_{0}^{2\pi}\!\frac{\cos\theta\,d\theta}{(1+e\cos\theta)^{2}}
=-\frac{\gamma_{1}-\gamma_{2}}{e},
\qquad\text{so that}\qquad
\gamma_{2}=\gamma_{1}+e\,\gamma_{1}' .
\]
Substituting $\gamma_{1}=2\pi(1-e^{2})^{-1/2}$ and its derivative
$\gamma_{1}'=2\pi e(1-e^{2})^{-3/2}$,
\[
\gamma_{2}=\frac{2\pi}{(1-e^{2})^{1/2}}+\frac{2\pi e^{2}}{(1-e^{2})^{3/2}}
=\frac{2\pi\bigl[(1-e^{2})+e^{2}\bigr]}{(1-e^{2})^{3/2}}
=\frac{2\pi}{(1-e^{2})^{3/2}} .
\]
Hence $c_{e}^{3}=2\pi/\gamma_{2}=(1-e^{2})^{3/2}$, that is,
\begin{equation}\label{eq:cnorm}
c_{e}=\sqrt{1-e^{2}},\qquad c_{e}'=-\frac{e}{c_{e}},\qquad
c_{e}''=-\frac{1}{c_{e}^{3}} ,
\end{equation}
where the prime denotes $d/de$. Requiring the period to be $2\pi$ fixes the
semi-major axis at $a=1$ by Kepler's third law, and $c_{e}$ is then the semi-minor
axis $b=a\sqrt{1-e^{2}}$. The same law, applied to the minimal
period $T=2\pi/k$ of the component $M_{k}$, reads $T=2\pi a^{3/2}$ and gives
$a=(T/2\pi)^{2/3}=k^{-2/3}$ there. Distinct values of $k$ thus give distinct
semi-major axes, which is why the $M_{k}$ are the connected components of the
critical set.

The general $2\pi$-periodic Kepler ellipse is obtained from the planar reference
ellipse $u_{e}$ by a rigid rotation and a phase shift. We write the rotation as a
product of three one-parameter rotations,
\begin{equation}\label{eq:M}
u_{e,\psi}(\tau+\varphi),\qquad
u_{e,\psi}=R(\psi)\,u_{e}(\tau),\qquad
R(\psi)=e^{\psi_{1}J_{1}}\,e^{\psi_{2}J_{2}}\,e^{\psi_{3}J_{3}}\in SO(3),
\end{equation}
parametrized by the eccentricity $e$, the three orientation (Euler) angles
$\psi=(\psi_{1},\psi_{2},\psi_{3})$, and the time phase $\varphi$. The ordering of
the three factors is chosen so that the innermost rotation $e^{\psi_{3}J_{3}}$ is
the rotation about the $z$-axis, which fixes $(0,0,1)$ and therefore acts
\emph{within} the reference plane $\{z=0\}$. The unit normal to the orbit plane is
then
\[
\nor=R(\psi_{1},\psi_{2})\,(0,0,1),\qquad
R(\psi_{1},\psi_{2}):=e^{\psi_{1}J_{1}}e^{\psi_{2}J_{2}},
\]
carried by the outer pair alone, and we call $(\psi_{1},\psi_{2})$ the normal
angles. The in-plane angle $\psi_{3}$ leaves the normal, and with it the orbit
plane, unchanged, and it is the argument of pericentre, measured from a reference
direction in the orbit plane that depends only on $\psi_{1}$ and $\psi_{2}$. This
ordering will be essential in Section~\ref{sec:circular}, where $\psi_{3}$ must
reduce to a phase shift along the circular orbit, and we adopt it consistently
from here on.

These curves form a five-dimensional manifold
\[
M=\{\,u_{e,\psi}(\,\cdot+\varphi)\,\},
\]
carrying the four geometric parameters $(e,\psi)$ together with the phase
$\varphi\in S^{1}$. The normal angles enter only
through $\nor$, which ranges over the sphere $S^{2}$. The pair $(e,\psi_{3})$
enters only through the planar vector $e\,(\cos\psi_{3},\sin\psi_{3})$, of which
they are the polar coordinates and which ranges over the open unit disc. The centre
of that disc is the circular orbit $e=0$, where the polar coordinates fail. That is
the subject of Section~\ref{sec:circular}.

\begin{proposition}\label{thm:kepler}
For every $e\in(0,1)$ the manifold $M$ is a nondegenerate critical manifold of
$\A_{0}$ at the planar Kepler ellipse $u_{e}$. The Hessian $\mathcal{L}_{e}:=\nabla^{2}\A_{0}(u_{e})$, acting on
$2\pi$-periodic vector fields, has kernel
\[
\ker\mathcal{L}_{e}=\operatorname{span}\{\,\partial_{\tau}u_{e},\ \partial_{e}u_{e},\
J_{1}u_{e},\ J_{2}u_{e},\ J_{3}u_{e}\,\},
\]
which is five-dimensional and coincides with the tangent space $T_{u_{e}}M$.
\end{proposition}

\begin{proof}
Differentiating $\nabla\A_{0}$ once more, the Hessian is the linear operator
\[
\mathcal{L}_{e}\eta=-\ddot\eta-\frac{\eta}{\abs{u_{e}}^{3}}
+3\,\frac{u_{e}\inner{u_{e}}{\eta}}{\abs{u_{e}}^{5}},
\]
acting on $2\pi$-periodic maps $\eta:\R\to\R^{3}$. This is a
second-order linear ODE for a map with values in $\R^{3}$, so its solutions,
periodic or not, form a space of dimension six, and the kernel of
$\mathcal{L}_{e}$ on $2\pi$-periodic fields is the subspace of the periodic ones
among them.

Write $u_{e}=(x,y,0)$. Five $2\pi$-periodic solutions of
$\mathcal{L}_{e}\eta=0$ are
\[
\partial_{\tau}u_{e}=\dot u_{e},\qquad
J_{3}u_{e}=(-y,x,0),\qquad
\partial_{e}u_{e},\qquad
J_{1}u_{e}=(0,0,y),\qquad J_{2}u_{e}=(0,0,x),
\]
the first two arising from the time-translation and in-plane
rotation invariance of $\A_{0}$, the third from the fact that the family
\eqref{eq:ellipse} consists of $2\pi$-periodic solutions for every $e$, so that
its $e$-derivative is a periodic Jacobi field, and the last two from the
invariance of $\A_{0}$ under rotations out of the plane. The last two take their
values in $\R_{z}$ and the first three in $\R^{2}_{xy}$, so the two groups are
independent of each other, and $J_{1}u_{e}$ and $J_{2}u_{e}$ are independent.
For $e\ne0$ the first three are independent, by the
following two observations. First, the angular
momentum $C(u)=u\times\dot u$, whose first variation at $u_{e}$ in the
direction of a vector field $\eta$ is
$\delta C[\eta]=\eta\times\dot u_{e}+u_{e}\times\dot\eta$, vanishes to first
order along the first two fields,
\[
\delta C[\dot u_{e}]
= u_{e}\times\ddot u_{e}
=u_{e}\times\Bigl(-\frac{u_{e}}{\abs{u_{e}}^{3}}\Bigr)=0,
\qquad
\delta C[J_{3}u_{e}]
=J_{3}\bigl(u_{e}\times\dot u_{e}\bigr)=c_{e}\,J_{3}(0,0,1)=0,
\]
the second by the identity $Ja\times b+a\times Jb=J(a\times b)$ for
$J\in\mathfrak{so}(3)$ (differentiate $e^{tJ}a\times e^{tJ}b=e^{tJ}(a\times b)$
at $t=0$), but not along the eccentricity direction. Differentiating
$u_{e}\times\dot u_{e}=c_{e}\,(0,0,1)$ in $e$ gives, by \eqref{eq:cnorm},
$\delta C[\partial_{e}u_{e}]=c_{e}'\,(0,0,1)\ne0$ for $0<e<1$. Hence if
$a\,\dot u_{e}+b\,J_{3}u_{e}+c\,\partial_{e}u_{e}=0$, applying the linear map
$\delta C$ leaves $c\,c_{e}'\rev{\,(0,0,1)}=0$, so $c=0$. Second, in the remaining relation
$a\,\dot u_{e}+b\,J_{3}u_{e}=0$, take the pointwise inner product with
$u_{e}$. Since $J_{3}u_{e}=(-y,x,0)$ is orthogonal to
$u_{e}=(x,y,0)$, this leaves
$a\,\inner{\dot u_{e}}{u_{e}}=\tfrac{a}{2}\,\tfrac{d}{d\tau}\abs{u_{e}}^{2}=0$.
For $e\ne0$ the radius $\abs{u_{e}}=c_{e}^{2}/(1+e\cos\theta)$ is nonconstant,
so $a=0$, and then $b=0$ because $\abs{J_{3}u_{e}}=\abs{u_{e}}>0$.

One further solution
arises from the scaling symmetry $u\mapsto\lambda^{2}u(\lambda^{-3}\tau)$ of the
Kepler equation, with generator
\[
\eta_{\mathrm{sc}}=2u_{e}-3\tau\,\dot u_{e},\qquad \mathcal{L}_{e}\eta_{\mathrm{sc}}=0 .
\]
This $\eta_{\mathrm{sc}}$ is \emph{not} periodic, because scaling changes the
period,
\[
\eta_{\mathrm{sc}}(\tau+2\pi)-\eta_{\mathrm{sc}}(\tau)=-6\pi\,\dot u_{e}\ne0 .
\]
The periodic solutions are therefore a proper subspace of the
six-dimensional solution space, of dimension at most five. The five fields above
are periodic and independent, so $\ker\mathcal{L}_{e}$ is five-dimensional with
the stated basis, and
this is precisely the tangent space to $M$ at $u_{e}$.
\end{proof}

\subsection{Lyapunov--Schmidt reduction and the continuation theorem}
\label{sec:cont}

By Proposition~\ref{thm:kepler} the whole of $M$ is a nondegenerate critical
manifold of $\A_{0}$.
The hypothesis under which we shall continue is local, namely that $\Phi$ has one
nondegenerate critical point, and it is verified at explicit points by the
computer-assisted proofs of Section~\ref{sec:cap}, never on all of $M$ at once.
A reduction performed globally along $M$ would therefore prove more than is used,
and would pay for it by carrying the geometry of $M$ through the functional
analysis, since one would need the kernel, the range and the projections of
$\nabla^{2}\A_{0}$ at every point of a noncompact five-dimensional manifold. Fixing
the base ellipse first and reducing only in a neighbourhood of it keeps the whole
argument attached to a single Hessian at a single point. We keep
the geometric picture in the statements, and the proof below does not use it.

We work in the Sobolev spaces $X=\Hp(\R^{3})$, the $2\pi$-periodic $H^{2}$ vector
fields, and $Y=\Lp(\R^{3})$, with the $\Lp$ inner product
$\inner{\cdot}{\cdot}_{\Lp}$ throughout. The pair is chosen so
that $\nabla\A_{H}$ is the pointwise Euler--Lagrange expression, a smooth map of
$X$ into $Y$ on the collision-free set, and so that the first variation is
represented by the $\Lp$ pairing, in which the projections below are taken. No
regularity is lost by it, since $\Hp$ embeds in $C^{1}$ and Step 4 shows the
critical points obtained to be smooth.

The continuation is governed by the restriction of the interaction term to the
Kepler manifold, which we name here and make explicit in
Section~\ref{sec:reduced},
\begin{equation}\label{eq:Phidef}
\Phi(e,\psi):=\A_{1}(u_{e,\psi}).
\end{equation}
It is a function of $(e,\psi)$ alone, with no dependence on the phase $\varphi$,
because $\A_{1}$ is invariant under time translation.

The hypothesis below is a property of one pair $(e_{0},\psi_{0})$,
and Theorem~\ref{thm:cap-ecc} verifies it at each pair listed in
Table~\ref{tab:cap-ecc}.

\begin{hypothesis}\label{hyp:reduced}
The eccentricity $e_{0}$ lies in $(0,1)$, the orientation
$\psi_{0}=(\psi_{0,1},\psi_{0,2},\psi_{0,3})$ has $\cos\psi_{0,2}\neq0$, and the
great circle of $\psi_{0}$ avoids the rotation axes of $H$, so that the $H$-orbit
of $u_{e_{0},\psi_{0}}$ is free of collisions. Moreover $(e_{0},\psi_{0})$ is a
nondegenerate critical point of $\Phi$, that is $\nabla\Phi(e_{0},\psi_{0})=0$ and
$\nabla^{2}\Phi(e_{0},\psi_{0})$ is invertible.
\end{hypothesis}

\begin{theorem}\label{thm:cont}
Let $(e_{0},\psi_{0})$ satisfy Hypothesis~\ref{hyp:reduced}. Then there exist
$\varepsilon_{0}>0$ and a $C^{1}$ branch $\varepsilon\mapsto u_{\varepsilon}\in
X$, defined for $0\le\varepsilon<\varepsilon_{0}$, of critical points of
$\A_{H}(\,\cdot\,;\varepsilon)$ with $u_{0}=u_{e_{0},\psi_{0}}$, unique up to
time translation, in the sense that every critical point of
$\A_{H}(\,\cdot\,;\varepsilon)$ in a fixed neighbourhood of $u_{0}$ in $X$ equals
$u_{\varepsilon}(\,\cdot+\varphi)$ for some $\varphi\in S^{1}$, so that the
solutions come in circles, one orbit of the time-translation action. Consequently the
symmetric $(n+1)$-body problem possesses a family of
$2\pi\varepsilon^{1/2}$-periodic solutions, carrying the symmetry \eqref{eq:symm}
of the group $H$, bifurcating from the Kepler ellipse $u_{e_{0},\psi_{0}}$.
Moreover the branch is asymptotic to the base ellipse,
\begin{equation}\label{eq:expansion}
u_{\varepsilon}=u_{e_{\varepsilon},\psi_{\varepsilon}}+O(\varepsilon)
\quad\text{in }X,
\qquad
(e_{\varepsilon},\psi_{\varepsilon})=(e_{0},\psi_{0})+O(\varepsilon),
\end{equation}
so that in particular $u_{\varepsilon}=u_{e_{0},\psi_{0}}+O(\varepsilon)$.
\end{theorem}

\begin{proof}
The four parameters are $e$ and $\psi=(\psi_{1},\psi_{2},\psi_{3})$, we write
$u_{e,\psi}$ for the curve \eqref{eq:M} at a phase fixed once and for all, and
$u_{0}=u_{e_{0},\psi_{0}}$. The phase $\varphi$ is not one of the four. It is
held fixed in $u_{e,\psi}$ and enters the argument as the fifth kernel direction
$\dot u_{0}$.
The proof is in
four steps. The first solves the equation in the directions transverse to the
kernel of the Hessian at $u_{0}$, the second shows that what remains is the
vanishing of the gradient of a function of $(e,\psi)$ alone, the third identifies
that function with $\Phi$ to leading order, and the fourth removes the
phase.

\emph{The kernel at the base point.} The Hessian
$\mathcal{L}_{0}=\nabla^{2}\A_{0}(u_{0})$
is self-adjoint as an unbounded operator on $Y$ with domain $X$, and Fredholm
of index $0$ from $X$ to $Y$, being
a relatively compact perturbation of $-\partial_{\tau}^{2}$. Its kernel follows
from Proposition~\ref{thm:kepler} by the group action. The action $\A_{0}$ is
invariant under phase shifts and under $O(3)$, and $u_{0}$ lies on the orbit of the
planar ellipse $u_{e_{0}}$ under that group, say $u_{0}=g\,u_{e_{0}}$. Hence
$\nabla^{2}\A_{0}(u_{0})=g\,\nabla^{2}\A_{0}(u_{e_{0}})\,g^{-1}$ and
$\ker\mathcal{L}_{0}=g\,\ker\nabla^{2}\A_{0}(u_{e_{0}})$, and $g$ takes the five generators of
Proposition~\ref{thm:kepler} to the same five generators at $u_{0}$, the three
rotational ones after the change of basis $J_{i}\mapsto R(\psi_{0})J_{i}R(\psi_{0})^{-1}$
of $\mathfrak{so}(3)$, which changes nothing, the span
$\{Ju_{0}:J\in\mathfrak{so}(3)\}$ not depending on the basis chosen in
$\mathfrak{so}(3)$. Therefore,
writing $K$ for the kernel,
\begin{equation}\label{eq:Kbasis}
K:=\ker\mathcal{L}_{0}=\operatorname{span}\{\,\dot u_{0},\ \partial_{e}u_{0},\
J_{1}u_{0},\ J_{2}u_{0},\ J_{3}u_{0}\,\}=T_{u_{0}}M,
\qquad \operatorname{ran}\mathcal{L}_{0}=K^{\perp},
\end{equation}
the second equality because a self-adjoint Fredholm operator has closed range
equal to the orthogonal complement of its kernel. We write $W:=K^{\perp}\cap X$, let
$Q$ denote the $\Lp$-orthogonal projection onto $K^{\perp}$, and let $P:=I-Q$ be
the projection onto $K$.

\emph{Step 1. The correction transverse to $K$.}
Both $\A_{0}$ and $\A_{1}$ are smooth on the open subset of $X$ on which $u$
and all the differences $u-Lu$, $L\neq I$, stay away from the origin, and $u_{0}$
lies in that subset by hypothesis, so $\nabla\A_{H}$ is a smooth map of a
neighbourhood of $u_{0}$ in $X$ into $Y$.
Consider
\[
N(e,\psi,w;\varepsilon):=Q\,\nabla\A_{H}(u_{e,\psi}+w;\varepsilon)
\in K^{\perp}\cap Y,
\qquad (e,\psi)\ \text{near}\ (e_{0},\psi_{0}),\quad w\in W .
\]
For every $(e,\psi)$ the curve $u_{e,\psi}$ is a critical point of $\A_{0}$, so
$N(e,\psi,0;0)=Q\,\nabla\A_{0}(u_{e,\psi})=0$. The derivative in $w$ at
$(e_{0},\psi_{0},0;0)$ is $\mathcal{L}_{0}$ restricted to $W$, and it is an
isomorphism onto
$K^{\perp}\cap Y$, being injective because $\ker\mathcal{L}_{0}=K$ and surjective
because $\operatorname{ran}\mathcal{L}_{0}=K^{\perp}$. The implicit function theorem
therefore gives $\varepsilon_{0}>0$, a neighbourhood of $(e_{0},\psi_{0})$, and a
unique small solution
\[
w=w(e,\psi,\varepsilon)\in W,\qquad
N\bigl(e,\psi,w(e,\psi,\varepsilon);\varepsilon\bigr)=0,
\]
of the same regularity as $\A_{H}$. Since $w(e,\psi,0)=0$ for every
$(e,\psi)$, the parameter derivatives of $w$ also vanish at $\varepsilon=0$, and
differentiability in $\varepsilon$ then gives $w=O(\varepsilon)$ together with its
derivatives in the parameters.

\emph{Step 2. The equations that remain are equivalent to
$\nabla\Phi_{\varepsilon}(e,\psi)=0$.} Put
\begin{equation}\label{eq:Phieps}
\Phi_{\varepsilon}(e,\psi)
:=\A_{H}\bigl(u_{e,\psi}+w(e,\psi,\varepsilon);\varepsilon\bigr).
\end{equation}
We claim that if $\nabla\Phi_{\varepsilon}(e,\psi)=0$ then
$u:=u_{e,\psi}+w(e,\psi,\varepsilon)$ is a critical
point of $\A_{H}(\,\cdot\,;\varepsilon)$.

By Step 1 the gradient $\nabla\A_{H}(u)$ lies in $K$. Write $\partial$ for any of
$\partial_{e},\partial_{\psi_{1}},\partial_{\psi_{2}},\partial_{\psi_{3}}$.
Differentiating \eqref{eq:Phieps}, and using that $\partial w$ lies in the
fixed subspace $W\subset K^{\perp}$ while $\nabla\A_{H}(u)$ lies in $K$, so that
the two pair to zero,
\[
\partial\,\Phi_{\varepsilon}(e,\psi)
=\inner{\nabla\A_{H}(u)}{\partial u_{e,\psi}+\partial w}_{\Lp}
=\inner{\nabla\A_{H}(u)}{\partial u_{e,\psi}}_{\Lp} .
\]
So $\nabla\Phi_{\varepsilon}(e,\psi)=0$ says that $\nabla\A_{H}(u)$ is orthogonal
to the four vectors $\partial u_{e,\psi}$. The fifth test
direction is $\dot u$, and it yields no equation, because $\A_{H}$ is invariant
under time translation,
\[
\inner{\nabla\A_{H}(u)}{\dot u}_{\Lp}
=\frac{d}{d\varphi}\Big|_{\varphi=0}\A_{H}\bigl(u(\cdot+\varphi)\bigr)=0 .
\]
Differentiating this identity in $u$ at a critical point gives
$\nabla^{2}\A_{H}(u)\dot u=0$, so $\nabla^{2}\A_{H}$ is singular at every critical
point and none of them is isolated. The phase is nonetheless neither an unknown of
$\Phi_{\varepsilon}$, which depends on $(e,\psi)$ alone, nor an equation, by the
identity above, and four equations in four unknowns are left.
It remains to see that those four vectors and $\dot u$
project onto a basis of $K$. At $(e_{0},\psi_{0},0)$ they are the four
parameter derivatives of \eqref{eq:M} at $u_{0}$ together with the velocity, and
these span $K$ as in \eqref{eq:Kbasis} precisely when \eqref{eq:M} is an immersion
at $(e_{0},\psi_{0})$, which is what the two hypotheses $e_{0}\in(0,1)$ and
$\cos\psi_{0,2}\neq0$ supply, by Remark~\ref{rem:immersion}.
Spanning after projection is an open condition, and the
five vectors depend continuously on $(e,\psi)$ and $\varepsilon$, so it persists
for $(e,\psi)$ near $(e_{0},\psi_{0})$ and $\varepsilon$ small. A vector $g\in
K$ satisfies $\inner{g}{v}_{\Lp}=\inner{g}{Pv}_{\Lp}$ for every $v$, so a vector
of $K$ orthogonal to five vectors whose projections span $K$ is zero, and
$\nabla\A_{H}(u)=0$, which proves the claim.
This is exactly what fails at the circular orbit, where
$\partial_{\psi_{3}}u_{0}$ and $\dot u_{0}$ coincide and only four independent
directions are produced.  

\emph{Step 3. The reduced function to leading order.} The value
$\A_{0}(u_{e,\psi})$ does not depend on the parameters, since
$\partial\,\A_{0}(u_{e,\psi})
=\inner{\nabla\A_{0}(u_{e,\psi})}{\partial u_{e,\psi}}_{\Lp}=0$. Call it $c_{0}$.
Writing $\mathcal{L}=\nabla^{2}\A_{0}(u_{e,\psi})$ and expanding in $w$,
\[
\A_{0}(u_{e,\psi}+w)=c_{0}+\inner{\nabla\A_{0}(u_{e,\psi})}{w}_{\Lp}
+\tfrac12\inner{\mathcal{L}w}{w}_{\Lp}+O(\norm{w}^{3})
=c_{0}+O(\varepsilon^{2}),
\]
the linear term vanishing because $u_{e,\psi}$ is critical for $\A_{0}$, and
$\varepsilon\A_{1}(u_{e,\psi}+w)=\varepsilon\Phi(e,\psi)
+O(\varepsilon^{2})$ by \eqref{eq:Phidef}. Differentiating the two expansions
once and twice in the parameters gives the same estimates for those derivatives,
since $w$ and its parameter derivatives are $O(\varepsilon)$. Hence
\begin{equation}\label{eq:redexpand}
\Phi_{\varepsilon}(e,\psi)=c_{0}+\varepsilon\,\Phi(e,\psi)
+\varepsilon^{2}G(e,\psi,\varepsilon),
\end{equation}
with a remainder $G$ that is $C^{2}$ in the parameters, uniformly in
$\varepsilon$ small. Dividing the gradient of \eqref{eq:redexpand} by
$\varepsilon$, the critical points of $\Phi_{\varepsilon}$ are for
$\varepsilon\neq0$ exactly the zeros of
$\nabla\Phi(e,\psi)+\varepsilon\nabla G(e,\psi,\varepsilon)$, a map that is
$C^{1}$ up to and including
$\varepsilon=0$. Its leading term $\Phi$, like $\Phi_{\varepsilon}$
itself, is a function of $(e,\psi)$ alone, so neither of them inherits the
singularity of $\nabla^{2}\A_{H}$ found in Step 2, and the invertibility of
$\nabla^{2}\Phi$ assumed in Hypothesis~\ref{hyp:reduced} is a condition on four
variables from which the phase is absent.

The map $\nabla\Phi+\varepsilon\nabla G$ vanishes at $(e_{0},\psi_{0},0)$, because
$\nabla\Phi(e_{0},\psi_{0})=0$ by hypothesis, and
its parameter derivative there is $\nabla^{2}\Phi(e_{0},\psi_{0})$, invertible by
hypothesis, so
the implicit function theorem yields a unique $C^{1}$ branch
$(e_{\varepsilon},\psi_{\varepsilon})\to(e_{0},\psi_{0})$ of critical points of
$\Phi_{\varepsilon}$. By Step 2,
$u_{\varepsilon}:=u_{e_{\varepsilon},\psi_{\varepsilon}}
+w(e_{\varepsilon},\psi_{\varepsilon},\varepsilon)$ is a critical
point of $\A_{H}(\,\cdot\,;\varepsilon)$, and undoing the rescaling
\eqref{eq:scale} gives the stated periodic solutions. The expansion
\eqref{eq:expansion} is the two statements just proved, namely
$w=O(\varepsilon)$ from Step 1 and
$(e_{\varepsilon},\psi_{\varepsilon})=(e_{0},\psi_{0})+O(\varepsilon)$, which
together give $u_{\varepsilon}=u_{e_{0},\psi_{0}}+O(\varepsilon)$.

\emph{Step 4. Uniqueness up to time translation.} Every critical point of
$\A_{H}(\,\cdot\,;\varepsilon)$ near $u_{0}$ is
smooth, with bounds depending only on the neighbourhood, by bootstrapping the
Euler--Lagrange equation \eqref{eq:EL1}, so time translation acts differentiably
on such points. The four parameters, the correction $w$ and the phase $\varphi$
together coordinatize a neighbourhood of $u_{0}$ in the set of critical points,
because the four directions $\partial u_{0}$ and the velocity $\dot u_{0}$ span
$K$ by Step 2, while $W$ is a complement of $K$. Any
critical point near $u_{0}$ is therefore $(u_{e,\psi}+w)(\,\cdot+\varphi)$ for one
choice of the three, and translating its phase away leaves a critical point
$u_{e,\psi}+w$, which Steps 1 and 3 identify as the one already constructed. So the
critical points near $u_{0}$ form the single circle traced by translating
$u_{\varepsilon}$ in time.
\end{proof}

\begin{remark}\label{rem:immersion}
The chart \eqref{eq:M}, phase included, is an immersion at $(e,\psi)$ if and
only if $e\neq0$ and $\cos\psi_{2}\neq0$, the axes of the three orientation
derivatives being $e_{1}$, $e^{\psi_{1}J_{1}}e_{2}$ and $\nor$, of determinant
$\cos\psi_{2}$. These are the two polar singularities of the chart, that of
$(e,\psi_{3})$ at the centre of the disc of eccentricity vectors and that of
$(\psi_{1},\psi_{2})$ at the pole of the sphere of normals, defects of the
coordinates and not of $M$. Neither is restrictive here, the smallest
$\abs{\cos\psi_{2}}$ over the rows of Tables~\ref{tab:cap-e0}
and~\ref{tab:cap-ecc} being $0.23$.
\end{remark}

\begin{remark}\label{rem:local}
Only the Hessian at the single ellipse $u_{0}$ entered the argument. Step 1 used
$\ker\mathcal{L}_{0}$ and $\operatorname{ran}\mathcal{L}_{0}$, and no property of
$\nabla^{2}\A_{0}$ at any other point of $M$ was required, the nondegeneracy at
neighbouring parameters being replaced by the openness of the condition in Step 2.
This is why Proposition~\ref{thm:kepler}, although stated for the whole family, is
used at one point at a time, once for each row of the tables of
Section~\ref{sec:cap}.
\end{remark}

\subsection{The reduced function $\Phi$ and its derivatives}\label{sec:reduced}

It remains to make the \emph{reduced function} $\Phi=\A_{1}|_{M}$ of
\eqref{eq:Phidef} explicit. It is a function of the eccentricity $e$ and the
spatial orientation $\psi=(\psi_{1},\psi_{2},\psi_{3})$ alone, and by the
$N_{O(3)}(H)$-invariance of $\A_{1}$ it is moreover invariant under the (finite)
action of the normalizer on $\psi$. Writing the interaction term explicitly,
\begin{equation}\label{eq:Phi}
\Phi(e,\psi)=\A_{1}|_{M}
=\frac12\int_{0}^{2\pi}\sum_{L\in H\setminus\{I\}}
\frac{1}{\abs{(I-L)u_{e,\psi}}}\,d\tau .
\end{equation}

\paragraph{Change of variables to the true anomaly.}
We re-express the time integral \eqref{eq:Phi} as an integral over the true
anomaly $\theta$. Decompose $u_{e,\psi}=r(\theta)\,R(\psi)\,\hat e(\theta)$ with
$R(\psi)=e^{\psi_{1}J_{1}}e^{\psi_{2}J_{2}}e^{\psi_{3}J_{3}}\in SO(3)$,
$\hat e(\theta)=(\cos\theta,\sin\theta,0)$ the in-plane unit vector, and $r$ as in
\eqref{eq:ellipse}. Since $\abs{(I-L)u_{e,\psi}}=r\,\abs{(I-L)R(\psi)\hat e}$, and
since $\dot\theta=c_{e}/r^{2}$ converts the time element to
$d\tau=r^{2}c_{e}^{-1}\,d\theta$, we obtain
\begin{equation}\label{eq:Phi2}
\Phi(e,\psi)=\int_{0}^{2\pi}W_{e}(\theta)\,S(\hat u)\,d\theta,
\qquad \hat u=R(\psi)\,\hat e(\theta),
\end{equation}
where $\hat u$ is the unit vector tracing the great circle
$R(\psi)\{\hat e(\theta)\}$ on
the sphere $S^{2}$, and the \emph{weight} $W_{e}$ and the geometric \emph{sum} $S$
are
\begin{equation}\label{eq:W}
W_{e}(\theta)=\frac{c_{e}}{2\,(1+e\cos\theta)},
\qquad\qquad
S(\hat u)=\sum_{L\ne I}\frac{1}{\abs{\hat u-L\hat u}} .
\end{equation}
Here $S$ is the purely geometric, $e$-independent interaction factor on the
sphere, and it measures how close the direction $\hat u$ comes to the rotation
axes of $H$. It is even, $S(-\hat u)=S(\hat u)$, because
$\abs{(-\hat u)-L(-\hat u)}=\abs{\hat u-L\hat u}$. The scalar weight $W_{e}$
collects all the radial and time-rescaling factors. The factor $r^{-1}$ from
$\abs{(I-L)u}^{-1}$ and the factor $r^{2}c_{e}^{-1}$ from $d\tau$ combine to the
net weight $r\,c_{e}^{-1}$, and with $r=c_{e}^{2}/(1+e\cos\theta)$ and the prefactor
$\tfrac12$ of \eqref{eq:Phi}, this is exactly $W_{e}$, in which
$c_{e}=\sqrt{1-e^{2}}$ by \eqref{eq:cnorm}. The integrand is singular precisely
when $L\hat u=\hat u$ for some $L\ne I$, that is, when the great circle of
$\hat u$ passes through a rotation axis of $H$. The admissible orientations $\psi$
are those whose great circle avoids all such axes, and on this open set $\Phi$ is
smooth.

\paragraph{The gradient.}
Only $\hat u$ depends on the orientation $\psi$, and only $W_{e}$ depends on the
eccentricity $e$, so differentiating \eqref{eq:Phi2} under the integral sign
yields
\[
\partial_{\psi_{j}}\Phi
=\int_{0}^{2\pi}W_{e}\,\partial_{\psi_{j}}S(\hat u)\,d\theta,
\qquad
\partial_{e}\Phi
=\int_{0}^{2\pi}\partial_{e}W_{e}\,S(\hat u)\,d\theta ,
\]
the eccentricity derivative falling entirely on the weight. With
$\partial_{\psi_{j}}S(\hat u)=\sum_{L\ne I}\partial_{\psi_{j}}
\abs{\hat u-L\hat u}^{-1}$ and
$\abs{\hat u-L\hat u}^{-1}=\inner{(I-L)\hat u}{(I-L)\hat u}^{-1/2}$,
\begin{equation}\label{eq:dpsi}
\partial_{\psi_{j}}\frac{1}{\abs{\hat u-L\hat u}}
=-\frac{\inner{(I-L)\hat u}{(I-L)\,\partial_{\psi_{j}}\hat u}}
{\abs{(I-L)\hat u}^{3}} .
\end{equation}
The orientation derivatives of $\hat u$ are obtained by inserting the appropriate
generator at the factor it differentiates in
$\hat u=e^{\psi_{1}J_{1}}e^{\psi_{2}J_{2}}e^{\psi_{3}J_{3}}\hat e(\theta)$,
\[
\partial_{\psi_{1}}\hat u=J_{1}\hat u,\qquad
\partial_{\psi_{2}}\hat u=e^{\psi_{1}J_{1}}J_{2}e^{\psi_{2}J_{2}}
e^{\psi_{3}J_{3}}\hat e,\qquad
\partial_{\psi_{3}}\hat u=e^{\psi_{1}J_{1}}e^{\psi_{2}J_{2}}J_{3}
e^{\psi_{3}J_{3}}\hat e .
\]
Only the outermost angle gives the clean form $\partial_{\psi_{1}}\hat
u=J_{1}\hat u$. The inner generators do not commute through the outer rotations,
so $\partial_{\psi_{2}}\hat u$ and $\partial_{\psi_{3}}\hat u$ are the nested
expressions above and must \emph{not} be replaced by $J_{2}\hat u,J_{3}\hat u$.
The weight derivative is explicit from \eqref{eq:cnorm}, using $c_{e}'=-e/c_{e}$,
\begin{equation}\label{eq:dW}
\partial_{e}W_{e}
=\frac{c_{e}'}{2(1+e\cos\theta)}-\frac{c_{e}\cos\theta}{2(1+e\cos\theta)^{2}}
=-\frac{e}{2c_{e}(1+e\cos\theta)}
-\frac{c_{e}\cos\theta}{2(1+e\cos\theta)^{2}} .
\end{equation}

\paragraph{The Hessian.}
A further differentiation gives the orientation block
\[
\partial_{\psi_{k}}\partial_{\psi_{j}}\Phi
=\int_{0}^{2\pi}W_{e}\,\partial_{\psi_{k}}\partial_{\psi_{j}}S(\hat u)\,d\theta,
\qquad
\partial_{\psi_{k}}\partial_{\psi_{j}}S(\hat u)
=\sum_{L\ne I}\partial_{\psi_{k}}\partial_{\psi_{j}}
\frac{1}{\abs{\hat u-L\hat u}},
\]
where, using
$\partial_{\psi_{k}}\abs{(I-L)\hat u}^{-3}
=-3\,\abs{(I-L)\hat u}^{-5}\inner{(I-L)\hat u}{(I-L)\partial_{\psi_{k}}\hat u}$,
\begin{align}\label{eq:hpsi}
-\partial_{\psi_{k}}\partial_{\psi_{j}}\frac{1}{\abs{\hat u-L\hat u}}
&=\frac{\inner{(I-L)\partial_{\psi_{k}}\hat u}{(I-L)\partial_{\psi_{j}}\hat u}}
{\abs{(I-L)\hat u}^{3}}
+\frac{\inner{(I-L)\hat u}{(I-L)\partial_{\psi_{k}}\partial_{\psi_{j}}\hat u}}
{\abs{(I-L)\hat u}^{3}}\nonumber\\
&\quad-3\,
\frac{\inner{(I-L)\hat u}{(I-L)\partial_{\psi_{j}}\hat u}\,
\inner{(I-L)\hat u}{(I-L)\partial_{\psi_{k}}\hat u}}
{\abs{(I-L)\hat u}^{5}} .
\end{align}
The second orientation derivatives again insert both generators at the factors
they differentiate, for instance
$\partial_{\psi_{1}}^{2}\hat u=J_{1}^{2}\hat u$ and
$\partial_{\psi_{2}}\partial_{\psi_{1}}\hat u
=J_{1}e^{\psi_{1}J_{1}}J_{2}e^{\psi_{2}J_{2}}e^{\psi_{3}J_{3}}\hat e$. The mixed
and pure eccentricity blocks move the $e$-derivatives onto $W_{e}$,
\[
\partial_{e}\partial_{\psi_{j}}\Phi
=\int_{0}^{2\pi}\partial_{e}W_{e}\,\partial_{\psi_{j}}S(\hat u)\,d\theta,
\qquad
\partial_{e}^{2}\Phi
=\int_{0}^{2\pi}\partial_{e}^{2}W_{e}\,S(\hat u)\,d\theta,
\]
with $\partial_{\psi_{j}}\abs{\hat u-L\hat u}^{-1}$ given by \eqref{eq:dpsi} and,
using $c_{e}''=-c_{e}^{-3}$,
\begin{equation}\label{eq:d2W}
\partial_{e}^{2}W_{e}
=\frac{c_{e}''}{2(1+e\cos\theta)}
-\frac{c_{e}'\cos\theta}{(1+e\cos\theta)^{2}}
+\frac{c_{e}\cos^{2}\theta}{(1+e\cos\theta)^{3}}
=-\frac{1}{2c_{e}^{3}(1+e\cos\theta)}
+\frac{e\cos\theta}{c_{e}(1+e\cos\theta)^{2}}
+\frac{c_{e}\cos^{2}\theta}{(1+e\cos\theta)^{3}} .
\end{equation}
These determine the gradient
$\nabla\Phi=(\partial_{e}\Phi,\partial_{\psi_{1}}\Phi,\partial_{\psi_{2}}\Phi,
\partial_{\psi_{3}}\Phi)$ and the Hessian $\nabla^{2}\Phi$ completely. Each entry
is a finite sum over $L\in H\setminus\{I\}$ of the explicit inner-product
expressions above, weighted by $W_{e}$, $\partial_{e}W_{e}$, or
$\partial_{e}^{2}W_{e}$, leaving only the single $\theta$-integral to be evaluated
in the computer-assisted verification of Hypothesis~\ref{hyp:reduced}.

Section~\ref{sec:cap} verifies the hypothesis at explicit points,
by interval arithmetic on the integral in $\theta$ and on the finite group sum, and
the resulting statement is Theorem~\ref{thm:cap-ecc}.

\section{The circular case, $e_{0}=0$}\label{sec:circular}

The chart \eqref{eq:M} degenerates at the circular orbit, which is why
Section~\ref{sec:cont} required $e_{0}>0$. The obstruction is one of coordinates
only. Section~\ref{ss:circ-chart} replaces the singular pair $(e,\psi_{3})$ by the
eccentricity vector and checks that $M$ is still a nondegenerate critical manifold
at the circle, Section~\ref{ss:circ-hessian} writes the gradient and the Hessian in
the new chart in terms of the integrals of Section~\ref{sec:reduced}, and
Section~\ref{ss:circ-cont} proves the continuation, Theorem~\ref{thm:circular}, the
counterpart of Theorem~\ref{thm:cont} at $e_{0}=0$.

\subsection{The chart singularity and the eccentricity-vector chart}
\label{ss:circ-chart}

We work with the chart \eqref{eq:M} of $M$, whose parameters are the eccentricity
$e$, the in-plane angle $\psi_{3}$, the two normal angles $(\psi_{1},\psi_{2})$ and
the phase $\varphi$. At the circle
$\partial_{\psi_{3}}u_{0}=\dot u_{0}$, so that chart is not an
immersion and Theorem~\ref{thm:cont} does not apply. The manifold itself is
unaffected.

The remedy is to replace the singular pair $(e,\psi_{3})$ by Cartesian
coordinates, the \emph{eccentricity vector}
\begin{equation}\label{eq:eps}
(\epsilon_{1},\epsilon_{2})\in\R^{2},\qquad
\epsilon_{1}=e\cos\psi_{3},\qquad \epsilon_{2}=e\sin\psi_{3},
\end{equation}
whose length $\sqrt{\epsilon_{1}^{2}+\epsilon_{2}^{2}}=e$ is the eccentricity and
whose direction $\arg(\epsilon_{1},\epsilon_{2})=\psi_{3}$ is the argument of
pericentre. The map $(e,\psi_{3})\mapsto(\epsilon_{1},\epsilon_{2})$ is the
polar-to-Cartesian map, so the singularity at the circle is the
ordinary singularity of polar coordinates at the origin, and the origin is the
circle $e=0$. At $e=0$ the angle $\psi_{3}$ only shifts time,
$\tau\mapsto\tau+\psi_{3}$, so without a compensating phase the rays
$\psi_{3}=\mathrm{const}$ would reach different curves as $e\to0$ and the chart
would not be defined at the origin. The parameters
$(\epsilon_{1},\epsilon_{2},\psi_{1},\psi_{2})$ therefore carry the curve
$u_{e,\psi}$ of \eqref{eq:M} with the phase $\varphi=-\psi_{3}$, which cancels
that shift. In the
chart $(\epsilon_{1},\epsilon_{2},\psi_{1},\psi_{2},\varphi)$ the manifold $M$ is
then regular through the circle, and the next proposition supplies, in that chart,
the basis of $K$ that Steps 2 and 4 of Theorem~\ref{thm:cont} require.

\begin{proposition}\label{prop:circ}
Let $u_{0}$ be the circular orbit of orientation $(\psi_{1},\psi_{2})$ with
$\cos\psi_{2}\neq0$, and let $\mathcal{L}_{0}=\nabla^{2}\A_{0}(u_{0})$. Then
\[
K:=\ker\mathcal{L}_{0}=\operatorname{span}\{\,\partial_{\epsilon_{1}}u_{0},\
\partial_{\epsilon_{2}}u_{0},\ \partial_{\psi_{1}}u_{0},\
\partial_{\psi_{2}}u_{0},\ \dot u_{0}\,\}=T_{u_{0}}M
\]
is five-dimensional. In particular $M$ is a nondegenerate critical manifold of
$\A_{0}$ at the circle, and the chart
$(\epsilon_{1},\epsilon_{2},\psi_{1},\psi_{2},\varphi)$ is an immersion there.
\end{proposition}

\begin{proof}
Rotating, take the orbit plane to be $\{z=0\}$, so that
$u_{0}(\tau)=(\cos\tau,\sin\tau,0)$. Proposition~\ref{thm:kepler} does not cover
the circle, so the kernel is counted directly. Here
$\abs{u_{0}}=1$, so
$\mathcal{L}_{0}\eta=-\ddot\eta-\eta+3u_{0}\inner{u_{0}}{\eta}$, a second-order
linear ODE for a map with values in $\R^{3}$, whose solutions, periodic or not,
form a space of dimension six. The scaling generator
$\eta_{\mathrm{sc}}=2u_{0}-3\tau\,\dot u_{0}$ solves it, as in the proof of
Proposition~\ref{thm:kepler}, and it is not periodic, since
$\eta_{\mathrm{sc}}(\tau+2\pi)-\eta_{\mathrm{sc}}(\tau)=-6\pi\,\dot u_{0}\ne0$.
The periodic solutions are therefore a proper subspace, of dimension at most
five.

The five listed vectors are periodic solutions, being derivatives of families of
$2\pi$-periodic critical points of $\A_{0}$, and they are independent. \rev{In the plane} they are
$\dot u_{0}$ together with
\[
\partial_{\epsilon_{1}}u_{0}=-\cos\tau\,u_{0}+2\sin\tau\,\dot u_{0},\qquad
\partial_{\epsilon_{2}}u_{0}=-\sin\tau\,u_{0}-2\cos\tau\,\dot u_{0},
\]
the first of these by differentiating
\eqref{eq:ellipse} at $e=0$, where $c_{e}'=0$ gives $\partial_{e}r=-\cos\tau$ and
$\partial_{e}\theta=2\sin\tau$, and the second by rotating that answer through
$\pi/2$. At the circle $u_{0}$ and $\dot u_{0}$ are orthonormal at
every instant, so a relation among the three has vanishing $u_{0}$ component,
which forces the coefficients of the two eccentricity derivatives to vanish and
then that of $\dot u_{0}$. \rev{Out of the plane} they are $\partial_{\psi_{1}}u_{0}$ and
$\partial_{\psi_{2}}u_{0}$, whose out-of-plane components are independent
exactly when $\cos\psi_{2}\neq0$ by Remark~\ref{rem:immersion}. Five
independent periodic solutions in a space of dimension at most five are a basis of
it, so $\dim\ker\mathcal{L}_{0}=5$.
\end{proof}

\begin{definition}[The forms $\Phi$ and $\Psi$, and the working notation]\label{def:PsiPhi}
The change of variables \eqref{eq:eps} acts only on the shape pair, replacing
$(e,\psi_{3})$ by $(\epsilon_{1},\epsilon_{2})$ and leaving the normal angles
$(\psi_{1},\psi_{2})$ untouched. We name the reduced function $\Phi=\A_{1}|_{M}$ of
\eqref{eq:Phi2} in each chart, writing $\Phi$ in the polar variables and
$\Psi$ in the Cartesian ones, by
\begin{equation}\label{eq:PsiPhi}
\Psi(\epsilon_{1},\epsilon_{2},\psi_{1},\psi_{2})
:=\Phi(e,\psi_{1},\psi_{2},\psi_{3}),
\qquad
e=\sqrt{\epsilon_{1}^{2}+\epsilon_{2}^{2}},\quad
\psi_{3}=\arg(\epsilon_{1},\epsilon_{2}).
\end{equation}
They are one and the same function on $M$, written in two coordinate systems, and
only $\Psi$ is smooth at the circle.
\end{definition}
\emph{Working notation.} The normal angles $(\psi_{1},\psi_{2})$ are common to $\Phi$
and $\Psi$ and stay fixed throughout, since they label the orbit plane. We therefore
omit them from the notation of \emph{both} functions, writing $\Phi(e,\psi_{3})$ and
$\Psi(\epsilon_{1},\epsilon_{2})$, and we record where a derivative is evaluated by
listing the remaining arguments. Thus $\partial_{e}\Phi(0,\psi_{3})$ is
$\partial_{e}\Phi$ at eccentricity $e=0$ and in-plane angle $\psi_{3}$, and
$\nabla_{\epsilon}\Psi(0,0)$ denotes the Cartesian gradient of $\Psi$ in the
eccentricity variables at the circle.

\subsection{From the polar to the Cartesian chart}
\label{ss:circ-hessian}

Although the value $\Phi(0,\psi_{3})$ is independent of $\psi_{3}$, a circle
having no distinguished pericentre, its derivatives at $e=0$ need not be. The chain
rule for \eqref{eq:eps}, applied at fixed $\psi_{3}$ to
$\Phi(e,\psi_{3})=\Psi(e\cos\psi_{3},e\sin\psi_{3})$, gives
\begin{equation}\label{eq:dict}
\partial_{e}\Phi=\cos\psi_{3}\,\partial_{\epsilon_{1}}\Psi
+\sin\psi_{3}\,\partial_{\epsilon_{2}}\Psi,
\end{equation}
that is, the $e$-derivative of $\Phi$ is the derivative of $\Psi$ along the ray
that leaves the origin in the direction $(\cos\psi_{3},\sin\psi_{3})$. The
Cartesian derivatives of $\Psi$ are recovered from these radial ones alone, by
varying $\psi_{3}$, and no new integral appears. What we assemble in this way is
the gradient and the Hessian of $\Psi$ at the circle in the four bifurcation
variables $(\epsilon_{1},\epsilon_{2},\psi_{1},\psi_{2})$.

That ray is $(\epsilon_{1},\epsilon_{2})=e\,(\cos\psi_{3},\sin\psi_{3})$, on which
$e$ is the arclength and the direction $(\cos\psi_{3},\sin\psi_{3})$ is constant.
Differentiating $\Psi$ along it once and twice by \eqref{eq:dict}, and evaluating
at the circle $e=0$ where all the rays meet, we obtain the three identities
\begin{align}
\partial_{e}\Phi(0,\psi_{3})
&=\cos\psi_{3}\,\partial_{\epsilon_{1}}\Psi(0,0)
+\sin\psi_{3}\,\partial_{\epsilon_{2}}\Psi(0,0),\label{eq:dir1}\\
\partial_{e}^{2}\Phi(0,\psi_{3})
&=\cos^{2}\psi_{3}\,\partial_{\epsilon_{1}}^{2}\Psi(0,0)
+2\cos\psi_{3}\sin\psi_{3}\,\partial_{\epsilon_{1}}\partial_{\epsilon_{2}}\Psi(0,0)
+\sin^{2}\psi_{3}\,\partial_{\epsilon_{2}}^{2}\Psi(0,0),\label{eq:dir2}\\
\partial_{e}\partial_{\psi_{j}}\Phi(0,\psi_{3})
&=\cos\psi_{3}\,\partial_{\epsilon_{1}}\partial_{\psi_{j}}\Psi(0,0)
+\sin\psi_{3}\,\partial_{\epsilon_{2}}\partial_{\psi_{j}}\Psi(0,0),
\qquad j=1,2.\label{eq:dir3}
\end{align}
The left-hand sides are the $e$-derivatives of $\Phi$ from
Section~\ref{sec:reduced}, with integrands \eqref{eq:dW}, \eqref{eq:d2W} and
\eqref{eq:dpsi}, evaluated at eccentricity $e=0$ and in-plane angle $\psi_{3}$, on
which they genuinely depend. The right-hand sides are the Cartesian quantities we
are after, namely the gradient entries $\partial_{\epsilon_{i}}\Psi(0,0)$, the
eccentricity Hessian entries
$\partial_{\epsilon_{i}}\partial_{\epsilon_{k}}\Psi(0,0)$, and the mixed entries
$\partial_{\epsilon_{i}}\partial_{\psi_{j}}\Psi(0,0)$.

The right-hand sides are linear in $(\cos\psi_{3},\sin\psi_{3})$ in
\eqref{eq:dir1} and \eqref{eq:dir3}, and quadratic in it in \eqref{eq:dir2}.
Evaluating the three identities at $\psi_{3}=0,\tfrac\pi2,\tfrac\pi4$, where
$(\cos\psi_{3},\sin\psi_{3})$ equals $(1,0)$, $(0,1)$ and
$\tfrac1{\sqrt2}(1,1)$, inverts them entry by entry. The \emph{gradient} entries follow from \eqref{eq:dir1} at
$\psi_{3}=0,\tfrac\pi2$,
\begin{equation}\label{eq:grad-recon}
\partial_{\epsilon_{1}}\Psi(0,0)=\partial_{e}\Phi(0,0),
\qquad
\partial_{\epsilon_{2}}\Psi(0,0)=\partial_{e}\Phi(0,\tfrac\pi2).
\end{equation}
For the \emph{eccentricity Hessian}, the quadratic form \eqref{eq:dir2} rewrites,
by the double-angle identities, as
\[
\partial_{e}^{2}\Phi(0,\psi_{3})
=\frac{\partial_{\epsilon_{1}}^{2}\Psi+\partial_{\epsilon_{2}}^{2}\Psi}{2}
+\frac{\partial_{\epsilon_{1}}^{2}\Psi-\partial_{\epsilon_{2}}^{2}\Psi}{2}\cos2\psi_{3}
+\partial_{\epsilon_{1}}\partial_{\epsilon_{2}}\Psi\,\sin2\psi_{3}
\qquad(\text{all at }(0,0)),
\]
so the three values of $\psi_{3}$ give
\begin{equation}\label{eq:hess-recon}
\partial_{\epsilon_{1}}^{2}\Psi(0,0)=\partial_{e}^{2}\Phi(0,0),\quad
\partial_{\epsilon_{2}}^{2}\Psi(0,0)=\partial_{e}^{2}\Phi(0,\tfrac\pi2),\quad
\partial_{\epsilon_{1}}\partial_{\epsilon_{2}}\Psi(0,0)
=\partial_{e}^{2}\Phi(0,\tfrac\pi4)
-\tfrac12\bigl(\partial_{\epsilon_{1}}^{2}\Psi+\partial_{\epsilon_{2}}^{2}\Psi\bigr).
\end{equation}
The \emph{mixed entries} between the eccentricity and the two normal angles follow
from \eqref{eq:dir3} at $\psi_{3}=0,\tfrac\pi2$, for $j=1,2$,
\begin{equation}\label{eq:mix-recon}
\partial_{\epsilon_{1}}\partial_{\psi_{j}}\Psi(0,0)
=\partial_{e}\partial_{\psi_{j}}\Phi(0,0),
\qquad
\partial_{\epsilon_{2}}\partial_{\psi_{j}}\Psi(0,0)
=\partial_{e}\partial_{\psi_{j}}\Phi(0,\tfrac\pi2).
\end{equation}
Finally the \emph{orientation entries} need no translation, because
$\psi_{1},\psi_{2}$ are coordinates in both charts, so
$\partial_{\psi_{k}}\partial_{\psi_{j}}\Psi=\partial_{\psi_{k}}\partial_{\psi_{j}}
\Phi$ ($j,k=1,2$), read directly from \eqref{eq:hpsi} at $e=0$ with the weight
$W_{e}(\theta)\big|_{e=0}=\tfrac12$.

Writing $x=(x_{1},x_{2},x_{3},x_{4})=(\epsilon_{1},\epsilon_{2},\psi_{1},\psi_{2})$
for the four regular variables, the Hessian of $\Psi$ at the circle is
\begin{equation}\label{eq:H4}
\nabla^{2}\Psi(0,0)=
\Bigl(\partial_{x_{i}}\partial_{x_{j}}\Psi(0,0)\Bigr)_{i,j=1,\dots,4} .
\end{equation}
Its entries with $i,j\le2$ are given by \eqref{eq:hess-recon}, those with $i\le2<j$
or $j\le2<i$ by \eqref{eq:mix-recon}, and those with $i,j\ge3$ by \eqref{eq:hpsi}
at $e=0$. Every entry is thus one of the Section~\ref{sec:reduced} integrals, taken
at $e=0$ and evaluated at $\psi_{3}\in\{0,\tfrac\pi2,\tfrac\pi4\}$, and the
verification of Section~\ref{sec:cap} evaluates \eqref{eq:H4} as a whole.

\paragraph{The eccentricity gradient vanishes at every orientation.}
Both entries of \eqref{eq:grad-recon} are values of $\partial_{e}\Phi(0,\psi_{3})$,
which we now compute. Putting $e=0$, $c_{0}=1$ and $c_{0}'=0$ into \eqref{eq:W}
and \eqref{eq:dW},
\begin{equation}\label{eq:Wzero}
W_{e}(\theta)\big|_{e=0}=\tfrac12,\qquad
\partial_{e}W_{e}(\theta)\big|_{e=0}=-\tfrac12\cos\theta ,
\end{equation}
and since $S(\hat u)$ does not depend on $e$, the $e$-derivative of
\eqref{eq:Phi2} falls entirely on the weight,
\begin{equation}\label{eq:dePhi0}
\partial_{e}\Phi(0,\psi_{3})
=\int_{0}^{2\pi}\partial_{e}W_{e}\big|_{e=0}\,S(\hat u(\theta))\,d\theta
=-\frac12\int_{0}^{2\pi}\cos\theta\;S(\hat u(\theta))\,d\theta .
\end{equation}
Shifting $\theta$ by $\pi$ replaces $\hat u$ by $-\hat u$, which leaves $S$
unchanged since it is even, and replaces $\cos\theta$ by $-\cos\theta$. The
integral is then equal to its own negative, hence zero, and with it
$\partial_{e}\Phi(0,\psi_{3})$ at every in-plane angle. By \eqref{eq:grad-recon},
\emph{with no symmetry hypothesis whatsoever},
\begin{equation}\label{eq:gradeps}
\nabla_{\epsilon}\Psi(0,0)
=\bigl(\partial_{\epsilon_{1}}\Psi,\partial_{\epsilon_{2}}\Psi\bigr)(0,0)=0 .
\end{equation}
Every circular orbit is therefore critical in the eccentricity directions,
whatever its orientation, and only the two orientation equations remain.

\begin{remark}\label{rem:block}
Since \eqref{eq:gradeps} holds at every orientation, differentiating it in
$\psi_{j}$ gives $\partial_{\epsilon_{i}}\partial_{\psi_{j}}\Psi(0,0)=0$, so the
mixed entries \eqref{eq:mix-recon} vanish and the Hessian \eqref{eq:H4} is block
diagonal at the circle. Nothing below uses this, the verification of
Section~\ref{sec:cap} building \eqref{eq:H4} as a whole.
\end{remark}

\subsection{The continuation at the circle}
\label{ss:circ-cont}

With $M$ a nondegenerate critical manifold at the circle, a chart that is
regular there, and the Hessian \eqref{eq:H4} written explicitly in that chart, the
continuation of Section~\ref{sec:cont} remains valid, and the only input it still
requires is a nondegenerate critical point of $\Psi$. By \eqref{eq:gradeps} the two
eccentricity components of the gradient of $\Psi$ vanish at every orientation, so
that only the two orientation equations remain, and we state the hypothesis in
that form.

\begin{hypothesis}\label{hyp:circular}
The orientation $\psi_{0}$ has $\cos\psi_{0,2}\neq0$ and a great
circle avoiding the rotation axes of $H$. Moreover the
circle $(\epsilon_{1},\epsilon_{2})=(0,0)$ of that orientation is a nondegenerate
critical point of $\Psi$, that is,
$\partial_{\psi_{1}}\Psi=\partial_{\psi_{2}}\Psi=0$ there, which together with
\eqref{eq:gradeps} makes $\nabla\Psi=0$ in all four variables
$(\epsilon_{1},\epsilon_{2},\psi_{1},\psi_{2})$, and the $4\times4$ Hessian
\eqref{eq:H4} in those coordinates is invertible.
\end{hypothesis}

As with Hypothesis~\ref{hyp:reduced}, this is a property of a
single orientation, and it is verified at each of the orientations listed in
Table~\ref{tab:cap-e0} by Theorem~\ref{thm:cap-circular}.

\begin{theorem}\label{thm:circular}
Under Hypothesis~\ref{hyp:circular}, the conclusion of Theorem~\ref{thm:cont}
holds at $e_{0}=0$. For all small $\varepsilon>0$ there is a $C^{1}$ branch of
symmetric, $2\pi\varepsilon^{1/2}$-periodic solutions bifurcating, as
$\varepsilon\to0$, from the circular orbit of orientation $\psi_{0}$,
again unique up to time translation, so that the solutions come in circles,
one orbit of that action. The
expansion \eqref{eq:expansion} holds verbatim, now in the eccentricity-vector
chart,
\[
u_{\varepsilon}=u_{\epsilon_{\varepsilon},\psi_{\varepsilon}}
+O(\varepsilon)\quad\text{in }X,
\qquad
(\epsilon_{\varepsilon},\psi_{\varepsilon})=(0,0,\psi_{0})+O(\varepsilon).
\]
\end{theorem}

\begin{proof}
The proof of Theorem~\ref{thm:cont} is run at the base point $u_{0}$, the
circular orbit of orientation $\psi_{0}$, with the four parameters
$(\epsilon_{1},\epsilon_{2},\psi_{1},\psi_{2})$ of the eccentricity-vector chart
\eqref{eq:eps} in place of $(e,\psi)$ and with $\Psi$ in place of $\Phi$. Two
things have to be checked, and nothing else.

Steps 2 and 4 need the four parameter derivatives
$\partial_{\epsilon_{1}}u_{0}$, $\partial_{\epsilon_{2}}u_{0}$,
$\partial_{\psi_{1}}u_{0}$, $\partial_{\psi_{2}}u_{0}$ together with $\dot u_{0}$
to be a basis of $K$. This is
Proposition~\ref{prop:circ}, whose hypothesis $\cos\psi_{0,2}\neq0$ is part of
Hypothesis~\ref{hyp:circular}. It is what fails in the Euler chart, where
$\partial_{\psi_{3}}u_{0}=\dot u_{0}$ leaves only four independent directions, and
it is what the eccentricity-vector chart repairs. That failure is
one of coordinates, and is not the phase degeneracy of Step 2, which is a symmetry
of $\A_{H}$ present at every eccentricity and left untouched here.

Steps 1 and 3 need the circle
$(\epsilon_{1},\epsilon_{2})=(0,0)$ of orientation $\psi_{0}$ to be a nondegenerate
critical point of $\Psi$ with collision-free $H$-orbit. The two
eccentricity components of $\nabla\Psi$ vanish there, and at every orientation, by
\eqref{eq:gradeps}, the two orientation components vanish and the $4\times4$
Hessian \eqref{eq:H4} is invertible by Hypothesis~\ref{hyp:circular}, and the
great circle of $\psi_{0}$ avoids the rotation axes of $H$ by the same hypothesis.

The conclusion, the expansion and the uniqueness up to time translation are then
those of Theorem~\ref{thm:cont}.
\end{proof}

\section{Computer-assisted verification of the critical points}\label{sec:cap}

In this section we outline the methods used to verify Hypotheses~\ref{hyp:reduced} and~\ref{hyp:circular}. Both require rigorously solving a zero-finding problem, together with an invertibility verification. Using computer-assisted proofs to solve finite-dimensional 
 zero-finding problems is now standard material, see \cite{MR1420838,MR2807595,MR3444942,MR3971222,MR3990999, CoGaLe2021} for a few examples. In such cases the only analysis required to construct a proof is computing the derivative of the zero-finding problem, which has already been done in Sections~\ref{sec:reduced} and~\ref{sec:circular}.


\subsection{The Newton--Kantorovich theorem and rigorous integration}
\label{ss:cap-prelim}

We collect here the two ingredients on which the verifications of
Section~\ref{ss:cap-results} rest. The first is a finite-dimensional Newton–Kantorovich theorem, which turns a good numerical approximation of a zero into a proof that a true zero lies nearby. The second is a rigorous bound on the integrals that define our maps, without which the hypotheses of that theorem cannot be checked. We begin with the Newton–Kantorovich argument. \rev{The statement below is standard. The use of an injective operator $\cA$ together with the contraction bound on $I-\cA\,D\cF$ goes back to Krawczyk \cite{Kra69} and to the existence test of Moore \cite{Moo77}; the form given here, with an approximate zero $\bar x$ and explicit bounds $Y$ and $Z$, is that of Yamamoto \cite{Yam98}, and the survey \cite{Rum10} places it in its general framework. Packaging $Y$ and $Z$ into the radii polynomial \eqref{EQ: radii_polynomial} follows \cite{DLM07}.}

\begin{theorem}[\rev{see \cite{Kra69,Moo77,Yam98,Rum10}}]
\label{thm:Newton-Kantorovich type theorem}
    Let $\cN >0$, \rev{let $\bar{x} \in \R^\cN$ and $r_*>0$, and let $\cF:\R^\cN \to \R^\cN$ be continuously differentiable on $\overline{B_{r_*}(\bar{x})}$}. Let $\cA \in B(\R^\cN)$ be an invertible matrix. Let \rev{$Y$, $Z$} be nonnegative constants satisfying
    \begin{equation} \label{EQ: Y}
        \norm{\cA\cF(\bar{x})}_\infty \le Y,
    \end{equation}
    \begin{equation} \label{EQ: Z}
        \norm{I - \cA D \cF(c)}_\infty \le Z,\quad \textit{for all }c \in \overline{B_{r_*}(\bar{x})}.
    \end{equation}
    Define the radii polynomial by
    \begin{equation} \label{EQ: radii_polynomial}
        P(r) = Y - (1 - Z)r.
    \end{equation}
    If there exists $0 < r_0 \le r_*$ such that $P(r_0) < 0$, then there exists a unique $\tilde{x} \in \overline{B_{r_0}(\bar{x})}$
    such that $\mathcal{F}(\tilde{x}) = 0$.
\end{theorem}

\rev{The proof of Theorem~\ref{thm:Newton-Kantorovich type theorem} can be found in \cite{MR4879804}.}


The challenge is that our maps contain integrals, which are generically impossible to evaluate exactly. For the remainder of this section, we let $f:\R \to \R$ be a smooth $2\pi$-periodic function, which is the case for the integrands in this paper, and we present the method we settled on for evaluating
\[
\frac{1}{2\pi}\int_{0}^{2\pi} f(\theta) d\theta.
\]
The naive numerical method would be to split $[0,2\pi]$ into many subintervals and use interval arithmetic to bound $f$ on each one. In practice, this does not yield tight bounds. The key realization is that the above integral is by definition the average of $f$ and its zero-th Fourier coefficient, $a_0$. Indeed, since $f$ is smooth, it admits a Fourier series representation
\[
f(\theta) = \sum_{n \in \mathbb{Z}} a_n e^{i n \theta},
\]
where
\[
a_n = \frac{1}{2\pi} \int_{0}^{2\pi} f(\theta) e^{- i n \theta} d\theta, \quad \forall n \in \mathbb{Z}.
\]
The goal now becomes rigorously computing $a_0$. To begin, we fix $N$ to be a large power of 2 and consider sample points
\[
\theta_k = \frac{2k\pi}{N}, \quad k = -\frac{N}{2}, \dots, \frac{N}{2}-1.
\]
Then we sample $f$ at these points and apply the Discrete Fourier Transform (DFT) to obtain approximate Fourier coefficients $(\bar{a}_n)_{n \in \mathbb{Z}}$.
\[
\bar{a}_n = \frac{1}{N} \sum_{k = -\frac{N}{2}}^{\frac{N}{2}-1} f(\theta_k) e^{-in\theta_k}
\]
Since we only care about $\bar{a}_0$, this is simply taking the average of our sampled values. Then, the discrete Poisson summation formula (found in \rev{Chapter~6 of \cite{Dft_book}, see also \cite{MR4930548} for an application in the context of a computer-assisted proof}) gives
\[
\rev{
\bar{a}_n =  a_n } +  \sum_{j = 1}^{\infty} (a_{n+jN} + a_{n-jN}), \quad n = -\frac{N}{2}, \dots, \frac{N}{2}-1.
\]
Once again, we will only use the $n=0$ case to bound $| a_0 - \bar{a}_0|$. The last ingredient needed is a rigorous bound on $|a_n|$ that holds for all $n$. For this we require an additional assumption on $f$. When treated as a complex function, it must be analytic on a region $\Omega \subset \mathbb{C}$ of the form
\[
\Omega = \left\{z \in \mathbb{C} : \text{Re}(z) \in [0, 2\pi], \text{Im}(z) \in [-\rho, \rho] \right\}
\]
for some $\rho > 0$. By the periodic nature of $f$, this analyticity extends to the entire strip of the complex plane. The contour integral of \rev{$f(z)e^{-inz}$ around any closed contour contained in $\Omega$} is therefore $0$. For the case $n \ge 0$, we consider the subset of $\Omega$ which has non-positive imaginary part. Let $\Gamma = \Gamma_1 \cup \Gamma_2 \cup \Gamma_3 \cup \Gamma_4$ be the contour around this region in the clockwise direction, where $\Gamma_1$ lies on the real line. So then
\[
0 = \int_{\Gamma} f(z)e^{-inz} dz = \int_{\Gamma_1} f(z)e^{-inz} dz + \int_{\Gamma_2} f(z)e^{-inz} dz + \int_{\Gamma_3} f(z)e^{-inz} dz +\int_{\Gamma_4} f(z)e^{-inz} dz.
\]
By $2\pi$-periodicity of $f$, we have \rev{that} the two middle terms cancel and we are left with
\[
\int_{0}^{2\pi} f(\theta)e^{-in\theta} d\theta = - \int_{2\pi}^{0} f(\theta - i\rho)e^{-in(\theta - i\rho)} d\theta = e^{-n\rho}\int_{0}^{2\pi} f(\theta \rev{-} i\rho)e^{-in\theta} d\theta.
\]
Therefore
\[
|a_n| = \frac{1}{2\pi}\left| \int_{0}^{2\pi} f(\theta)e^{-in\theta} d\theta \right|
= \frac{e^{-n\rho}}{2\pi} \left|\int_{0}^{2\pi} f(\theta \rev{-} i\rho)e^{-in\theta} d\theta \right|
\leq \frac{e^{-n\rho}}{2\pi} \int_{0}^{2\pi} |f(\theta \rev{-} i\rho)| d\theta
\leq \nu^{-n} C_\rho,
\]
where $C_\rho$ is an upper bound on the integral expression and $\nu = e^\rho$. In practice we choose a very coarse and easy-to-compute upper bound $C_\rho$, for reasons that will be apparent below. For the case $n < 0$, one can obtain a similar expression by analysing the contour integral around the other half region of $\Omega$\rev{, giving $|a_{n}|\le\nu^{-\abs{n}}C_{\rho}$ for every $n\in\Z$}. Recalling the discrete Poisson summation formula above, we now have
\[
|a_0 - \bar{a}_0| = \left|\sum_{j = 1}^{\infty} (a_{jN} + a_{-jN}) \right|
\leq 2C_\rho \sum_{j = 1}^{\infty} \left(\nu^{-N}\right)^j
= 2C_\rho \frac{\nu^{-N}}{1 - \nu^{-N}}
= \frac{2 C_\rho}{\nu^N-1}.
\]
When $N$ is large, this error is small even with a poor upper bound $C_\rho$.


\subsection{Enclosures of the critical points, and the solutions they produce}
\label{ss:cap-results}

We now apply the scheme of Section~\ref{ss:cap-prelim} to the reduced function
itself. The outcome is recorded in the form in which the rest of the paper uses
it, as a verification of Hypotheses~\ref{hyp:reduced}
and~\ref{hyp:circular} at explicit points. There are two statements rather
than one because the two hypotheses are posed on different charts of the Kepler
manifold, and the unknown is not the same in the two cases.

\begin{enumerate}
\item[$\bullet$] Away from the circle, $e\neq0$, the chart is the Euler chart
      \eqref{eq:M} and the unknown is the quadruple
      $p=(e,\psi_{1},\psi_{2},\psi_{3})$. The map to be solved is
      $F=2\nabla\Phi:\R^{4}\to\R^{4}$, whose four components
      $\partial_{e}\Phi$, $\partial_{\psi_{1}}\Phi$, $\partial_{\psi_{2}}\Phi$,
      $\partial_{\psi_{3}}\Phi$ are the integrals of Section~\ref{sec:reduced},
      and its derivative $DF=2\nabla^{2}\Phi$ is the $4\times4$ matrix assembled
      from \eqref{eq:dW}, \eqref{eq:d2W}, \eqref{eq:dpsi}, and \eqref{eq:hpsi}.
\item[$\bullet$] At the circle, $e=0$, the chart is the eccentricity-vector chart
      \eqref{eq:eps} and the two eccentricity equations are already satisfied,
      identically and for every orientation, by \eqref{eq:gradeps}. The unknown is
      therefore only the pair of normal angles $p=(\psi_{1},\psi_{2})$, and the map
      to be solved is $F=2(\partial_{\psi_{1}}\Psi,\partial_{\psi_{2}}\Psi):
      \R^{2}\to\R^{2}$, evaluated at $(\epsilon_{1},\epsilon_{2})=(0,0)$. The
      nondegeneracy required by Hypothesis~\ref{hyp:circular} is that of the full
      $4\times4$ Hessian \eqref{eq:H4}, which the two-dimensional zero-finding
      problem does not see and which is therefore verified separately.
\end{enumerate}

In both cases the procedure is the same. We first compute a numerical
approximation $\bar p$ of the critical point. We then evaluate $F(\bar p)$ and
$DF(c)$ for $c$ in a ball about $\bar p$ in interval arithmetic, the
$\theta$-integrals by the Fourier method of Section~\ref{ss:cap-prelim} and the
group sum as the finite sum over the $\abs H-1$ elements of $H\setminus\{I\}$,
enumerated as the closure of the two generators $A,B$ of
Section~\ref{sec:groups}. From these we obtain the constants $Y$ and $Z$ of
\eqref{EQ: Y} and \eqref{EQ: Z}, and a radius $r_{0}$ at which the radii
polynomial \eqref{EQ: radii_polynomial} is negative, so that by
Theorem~\ref{thm:Newton-Kantorovich type theorem} a unique zero of $F$ lies
within $r_{0}$ of $\bar p$. Nondegeneracy is obtained on the same enclosure, by
bounding the determinant of the Hessian away from zero on
$\overline{B_{r_{0}}(\bar p)}$, so that it holds at the enclosed zero in
particular. The approximations $\bar p$ are those produced by the two Julia
notebooks \texttt{critical\_points\_simple.ipynb} (the circular case) and
\texttt{critical\_points\_simple5.ipynb} (the case $e\neq0$). Both are
self-contained, requiring only the standard library modules
\texttt{LinearAlgebra} and \texttt{Printf}, and both generate the group as the
closure of two axis-angle rotations, verified element by element against the
matrices $A,B$ of Section~\ref{sec:groups}, so that all coordinates below are in
the frame of this paper. The code, the validated bounds, and the data are
available at \cite{Constantineau_CAP_Code}. All angles are
in radians, reduced to $[0,2\pi)$.

\begin{theorem}[Circular critical points, $e_{0}=0$]
\label{thm:cap-circular}
Let the group $H$, the approximate orientation $\bar p=(\bar\psi_{1},\bar\psi_{2})$
and the radius $r_{0}$ be \rev{given by} one of the seven rows of Table~\ref{tab:cap-e0}, and let
$n=\abs H$. Then the map
$F=(\partial_{\psi_{1}}\Psi,\partial_{\psi_{2}}\Psi)$, evaluated at the circle
$(\epsilon_{1},\epsilon_{2})=(0,0)$, has a unique zero $\psi_{0}$ with
$\norm{\psi_{0}-\bar p}_\infty \le r_{0}$, the great circle of $\psi_{0}$ avoids the
rotation axes of $H$, \rev{$\cos\psi_{0,2}\neq0$,} and the $4\times4$ Hessian $\nabla^{2}\Psi$ of
\eqref{eq:H4} is invertible there. That is, Hypothesis~\ref{hyp:circular} holds
at $\psi_{0}$.

Consequently, by Theorem~\ref{thm:circular}, there is $\varepsilon_{0}>0$ such
that for every $0<\varepsilon<\varepsilon_{0}$ the $(n+1)$-body problem with
central mass $\varepsilon^{-1}$ has a $2\pi\varepsilon^{1/2}$-periodic solution of
the form \eqref{eq:solform}, with the expansion
\[
u_{\varepsilon}=u_{\epsilon_{\varepsilon},\psi_{\varepsilon}}+O(\varepsilon),
\qquad
(\epsilon_{\varepsilon},\psi_{\varepsilon})=(0,0,\psi_{0})+O(\varepsilon),
\]
about the circular Kepler orbit of orientation $\psi_{0}$.
\end{theorem}

\begin{proof}

\begin{table}[ht]
\begin{center}
\begin{tabular}{|c|c|c|c|c|c|c|c|}
\hline
$H$ & $\bar\psi_{1}$ & $\bar\psi_{2}$ & Y & Z & $r_{0}$ & $|\det2\nabla^{2}\Psi|$ & Figure \\ \hline
$\T$ & $5.497787$ & $5.004235$ & $ 1.3 \times 10^{-13} $ & $ 9.7 \times 10^{-8}$ & $1.3 \times 10^{-13}$ & $8.0 \times 10^{2}$ & \ref{fig:cap-tet-circ}
\\ \hline
$\mathbb{O}$ \#1 & $5.348523$ & $5.503986$ & $2.3 \times 10^{-14}$ & $3.4 \times 10^{-8}$ & $2.3 \times 10^{-14}$ & $2.8 \times 10^{5} $ &
\ref{fig:cap-oct-circ1} \\ \hline
$\mathbb{O}$ \#2 & $5.250347$ & $5.166846$ & $ 5.0 \times 10^{-14} $ & $ 5.3 \times 10^{-8} $ & $5.0 \times 10^{-14}$ & $ 1.0 \times 10^{5} $ &
\ref{fig:cap-oct-circ2} \\ \hline
$\mathbb{I}$ \#1 & $5.799017$ & $6.188895$ & $ 1.2 \times 10^{-14} $ & $ 6.1 \times 10^{-8} $ & $1.2 \times 10^{-14}$ & $ 1.2 \times 10^{7} $ &
\ref{fig:cap-ico-circ1} \\ \hline
$\mathbb{I}$ \#2 & $5.849764$ & $6.046874$ & $ 1.2 \times 10^{-14} $ & $ 4.0 \times 10^{-8} $ & $1.2 \times 10^{-14}$ & $ 4.2 \times 10^{7} $ &
\ref{fig:cap-ico-circ2} \\ \hline
$\mathbb{I}$ \#3 & $5.983287$ & $6.046539$ & $ 1.0\times 10^{-14} $ & $ 3.3 \times 10^{-8} $ & $1.0 \times 10^{-14}$ & $ 1.5 \times 10^{8} $ &
\ref{fig:cap-ico-circ3} \\ \hline
$\mathbb{I}$ \#4 & $6.191860$ & $5.915899$ & $ 1.4 \times 10^{-14} $ & $ 5.4 \times 10^{-8} $ & $1.4 \times 10^{-14}$ & $ 1.1 \times 10^{8} $ &
\ref{fig:cap-ico-circ4} \\ \hline
\end{tabular}
\end{center}
\caption{Critical orientations of the circular orbit, $e_{0}=0$,
enclosed by Theorem~\ref{thm:cap-circular}, each row yielding solutions by
Theorem~\ref{thm:circular}, rounded to 2 significant digits.}
\label{tab:cap-e0}
\end{table}



\rev{The enclosures are those produced by the procedure described above, run at the data of Table~\ref{tab:cap-e0}; the code and the validated bounds are available at \cite{Constantineau_CAP_Code}.} The consequence follows by citation. \rev{Four} assertions have just been established, the existence of the zero $\psi_{0}$ inside the enclosure, the avoidance of the rotation axes of $H$ by its great circle, \rev{the inequality $\cos\psi_{0,2}\neq0$,} and the invertibility of the $4\times4$ Hessian $\nabla^{2}\Psi$ at the circle of that orientation. These are
exactly the requirements of Hypothesis~\ref{hyp:circular} of
Section~\ref{sec:circular}. Theorem~\ref{thm:circular} of that section therefore
applies at $\psi_{0}$, and it supplies both the branch of
$2\pi\varepsilon^{1/2}$-periodic solutions with the symmetry \eqref{eq:symm} and
its expansion in powers of $\varepsilon$.
\end{proof}

\begin{theorem}[Eccentric critical points, $e_{0}\neq0$]
\label{thm:cap-ecc}
Let the group $H$, the approximate critical point
$\bar p=(\bar e,\bar\psi_{1},\bar\psi_{2},\bar\psi_{3})$ and the radius $r_{0}$ be
\rev{given by} one of the four rows of Table~\ref{tab:cap-ecc}, and let $n=\abs H$. Then the map
$F=\nabla\Phi$ has a unique zero
$(e_{0},\psi_{0})$ with $\norm{(e_{0},\psi_{0})-\bar p}_\infty \le r_{0}$. It satisfies
$e_{0}\in(0,1)$, its great circle avoids the rotation axes of $H$, \rev{$\cos\psi_{0,2}\neq0$,} and
$\nabla^{2}\Phi(e_{0},\psi_{0})$ is invertible. That is,
Hypothesis~\ref{hyp:reduced} holds at $(e_{0},\psi_{0})$.

Consequently, by Theorem~\ref{thm:cont}, there is $\varepsilon_{0}>0$ such that
for every $0<\varepsilon<\varepsilon_{0}$ the $(n+1)$-body problem with central
mass $\varepsilon^{-1}$ has a $2\pi\varepsilon^{1/2}$-periodic solution with the
symmetry \eqref{eq:symm}, with the expansion \eqref{eq:expansion} about the Kepler
ellipse $u_{e_{0},\psi_{0}}$ of eccentricity $e_{0}$ and orientation $\psi_{0}$.
In particular $u_{\varepsilon}=u_{e_{0},\psi_{0}}+O(\varepsilon)$.
\end{theorem}

\begin{proof}

\begin{table}[ht]
\begin{center}
\begin{tabular}{|c|c|c|c|c|c|c|c|c|}
\hline
$H$ & $\bar e$ & $\bar\psi_{1}$ & $\bar\psi_{2}$ & $\bar\psi_{3}$ & Y & Z & $r_{0}$ &
Figure \\ \hline
$\T$ & $0.844836$ & $5.497787$ & $4.947131$ & $3.141593$ & $9.2 \times 10^{-13}$ & $1.2 \times 10^{-6}$ & $9.2 \times 10^{-13}$ &
\ref{fig:cap-tet-ecc} \\ \hline
$\mathbb{O}$ \#1 & $0.873524$ & $5.348741$ & $5.200484$ & $4.100400$ & $1.6 \times 10^{-12}$ & $3.9 \times 10^{-6}$ &
$1.6 \times 10^{-12}$ & \ref{fig:cap-oct-ecc1} \\ \hline
$\mathbb{O}$ \#2 & $0.962477$ & $5.441475$ & $5.535369$ & $2.949510$ & $6.6 \times 10^{-13}$ & $4.2 \times 10^{-6}$ &
$6.6 \times 10^{-13}$ & \ref{fig:cap-oct-ecc2} \\ \hline
$\mathbb{I}$ & $0.679506$ & $5.793234$ & $6.190758$ & $6.259757$ & $7.8 \times 10^{-12}$ & $5.4 \times 10^{-5}$ &
$7.8 \times 10^{-12}$ & \ref{fig:cap-ico-ecc} \\ \hline
\end{tabular}
\end{center}
\caption{Critical points of positive eccentricity, enclosed by
Theorem~\ref{thm:cap-ecc}, each row yielding solutions by
Theorem~\ref{thm:cont}, rounded to 2 significant digits.}
\label{tab:cap-ecc}
\end{table}



\rev{
\begin{remark}
    In Table~\ref{tab:cap-ecc} we do not verify the invertibility of $\nabla^2 \Phi$ as $Z < 1$ in the CAP implies it. Indeed,
    \[
    Z < 1 \implies \norm{I - 2A\nabla^2 \Phi} <1 \implies 2A\nabla^2 \Phi \text{ invertible} \implies \nabla^2 \Phi \text{ invertible}.
    \] 
\end{remark}

}

\rev{The enclosures are those produced by the procedure described above, run at the data of Table~\ref{tab:cap-ecc}; the code and the validated bounds are available at \cite{Constantineau_CAP_Code}.} The consequence follows by citation. The assertions just established are the existence of the zero $(e_{0},\psi_{0})$ inside the enclosure, with $e_{0}\in(0,1)$, with a great circle avoiding the rotation axes of $H$, \rev{with $\cos\psi_{0,2}\neq0$,} and the invertibility of $\nabla^{2}\Phi(e_{0},\psi_{0})$. These are exactly the
requirements of Hypothesis~\ref{hyp:reduced} of Section~\ref{sec:cont}, that is,
$(e_{0},\psi_{0})$ is a nondegenerate critical point of the reduced function
$\Phi$. Theorem~\ref{thm:cont} of Section~\ref{sec:cont} therefore applies at
$(e_{0},\psi_{0})$, and it supplies both the branch of
$2\pi\varepsilon^{1/2}$-periodic solutions with the symmetry \eqref{eq:symm} and
its expansion \eqref{eq:expansion}.
\end{proof}

Two of the tetrahedral angles in Table~\ref{tab:cap-ecc} are exact,
$\psi_{1}=7\pi/4$ and $\psi_{3}=\pi$, both forced by the mirror
$\nor_{2}=\nor_{3}$ of the normalizer $N_{O(3)}(H)$ on which that critical
point lies, whereas the two octahedral points lie in the interior of a fundamental
domain of that normalizer.

The \rev{eleven} configurations enclosed by Theorems~\ref{thm:cap-ecc} and \ref{thm:cap-circular} are drawn in Figures~\ref{fig:cap-grid1} and \ref{fig:cap-grid2} below and, for the most eccentric
of them, in Figure~\ref{fig:intro-oct2} of the introduction.

\begin{figure}[htbp]
    \centering
    \begin{subfigure}[b]{0.45\textwidth}
        \centering
        \includegraphics[width=\textwidth]{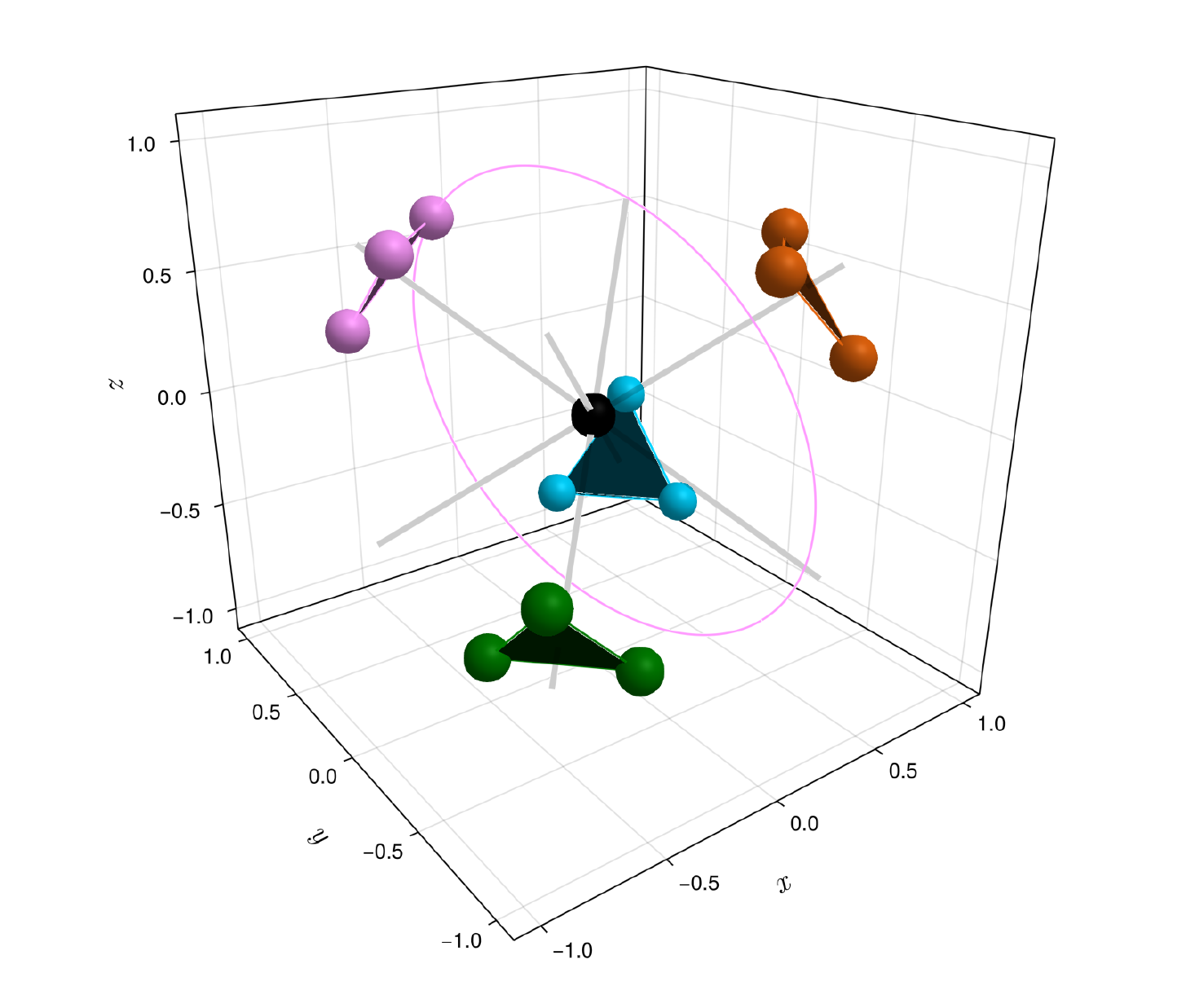}
        \caption{$\T$}
        \label{fig:cap-tet-circ}
    \end{subfigure}
    \hfill
    \begin{subfigure}[b]{0.45\textwidth}
        \centering
        \includegraphics[width=\textwidth]{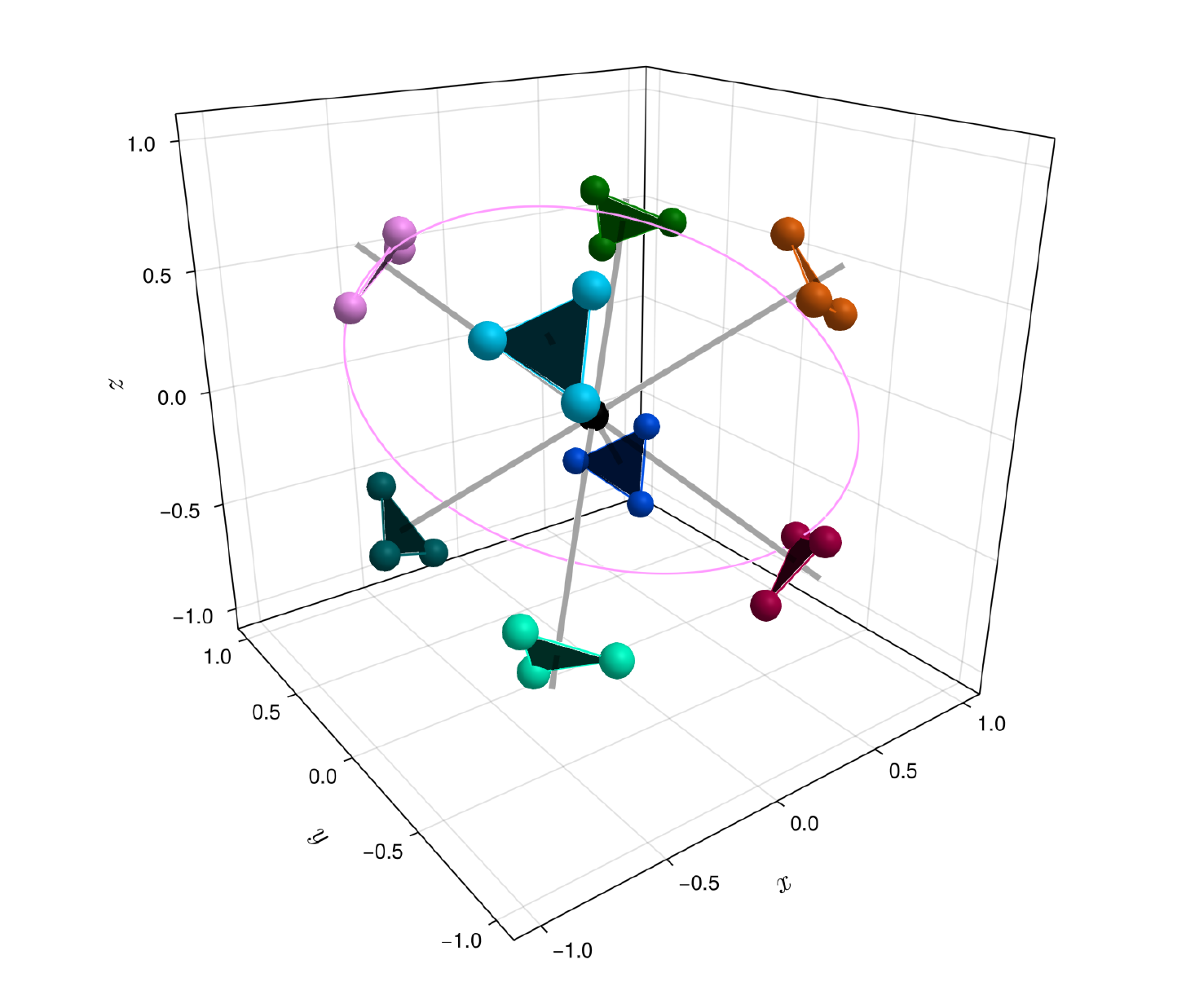}
        \caption{$\mathbb{O}$ \#1}
        \label{fig:cap-oct-circ1}
    \end{subfigure}
    
    \begin{subfigure}[b]{0.45\textwidth}
        \centering
        \includegraphics[width=\textwidth]{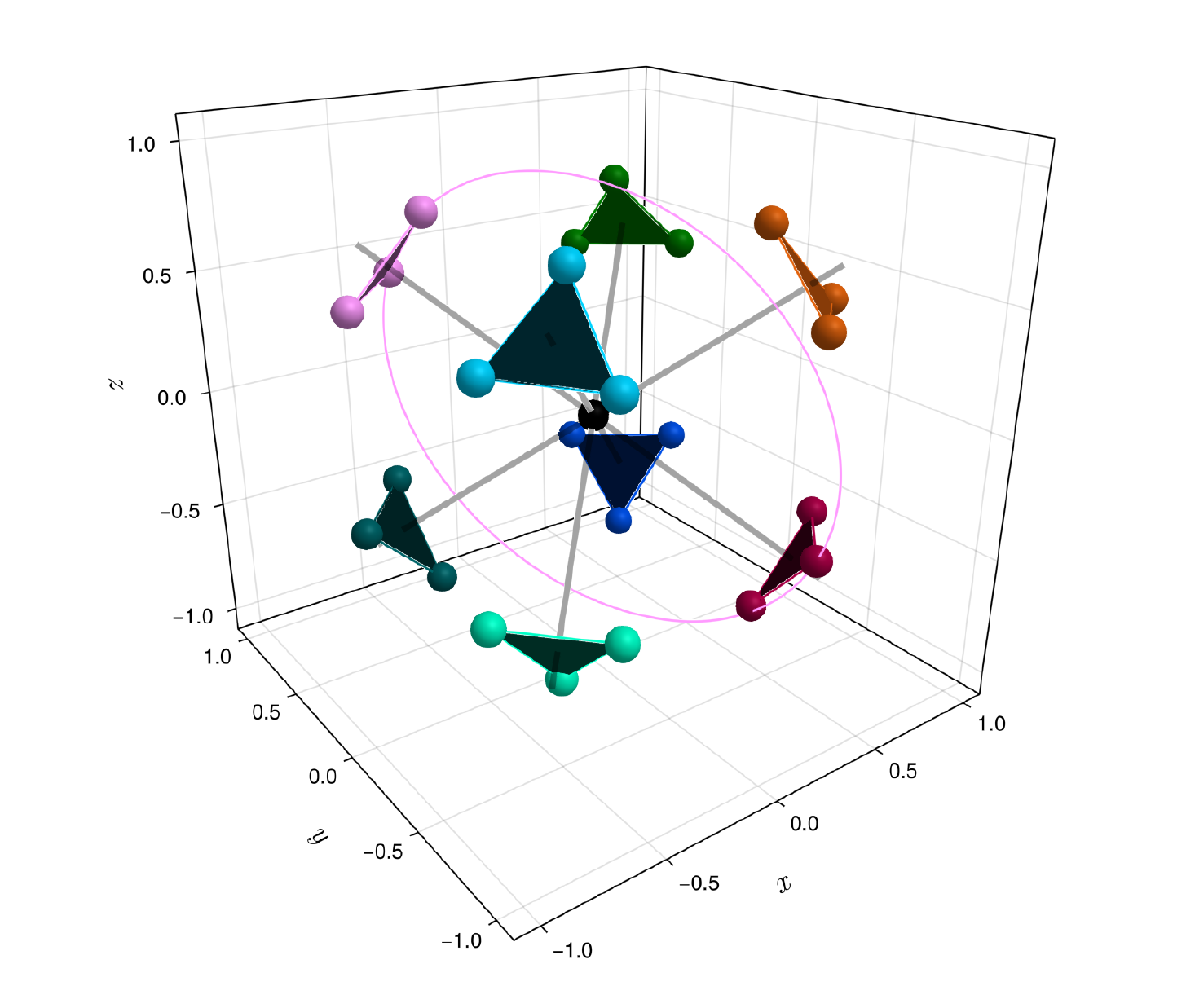}
        \caption{$\mathbb{O}$ \#2}
        \label{fig:cap-oct-circ2}
    \end{subfigure}
    \hfill
    \begin{subfigure}[b]{0.45\textwidth}
        \centering
        \includegraphics[width=\textwidth]{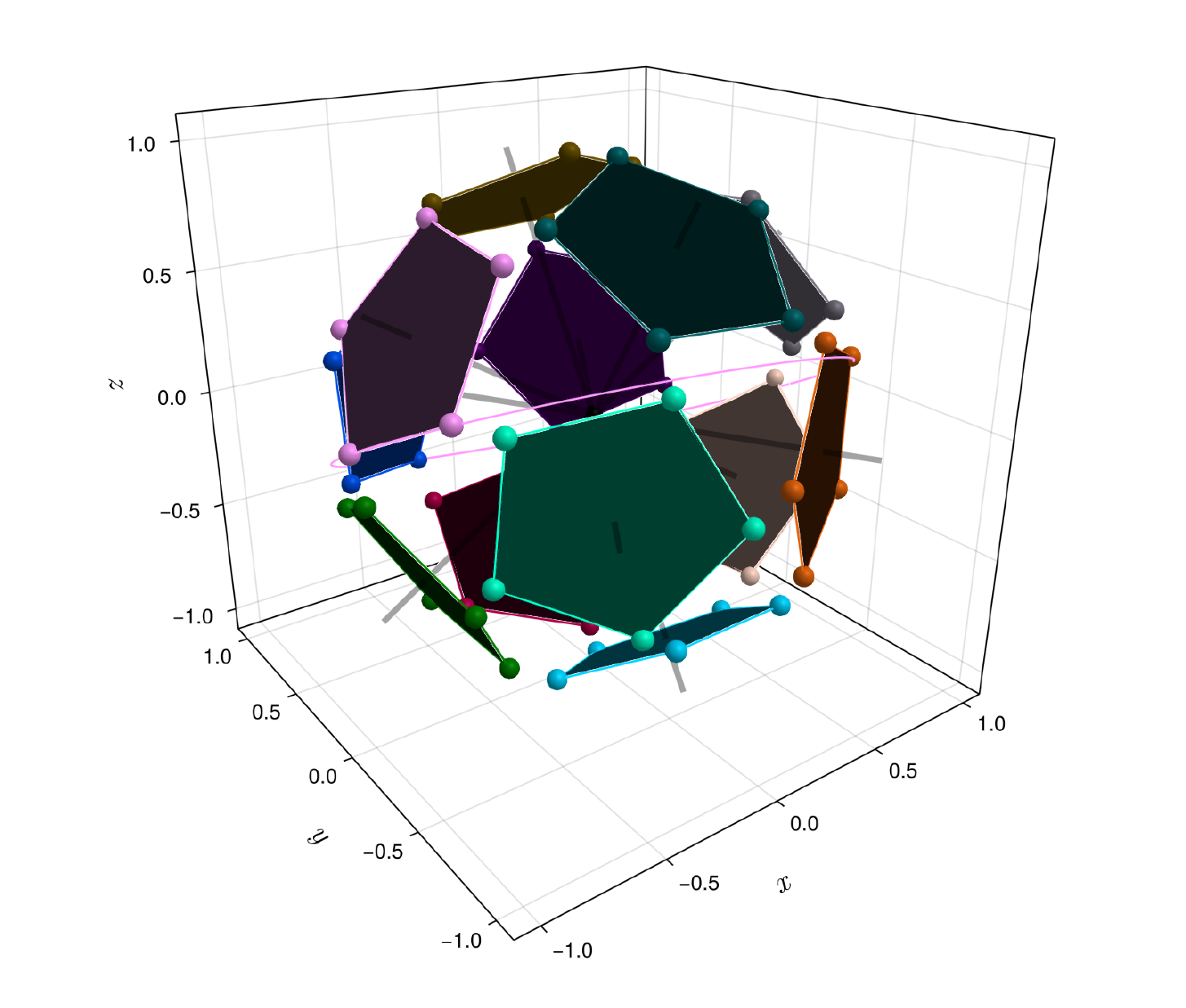}
        \caption{$\mathbb{I}$ \#1}
        \label{fig:cap-ico-circ1}
    \end{subfigure}

    \begin{subfigure}[b]{0.45\textwidth}
        \centering
        \includegraphics[width=\textwidth]{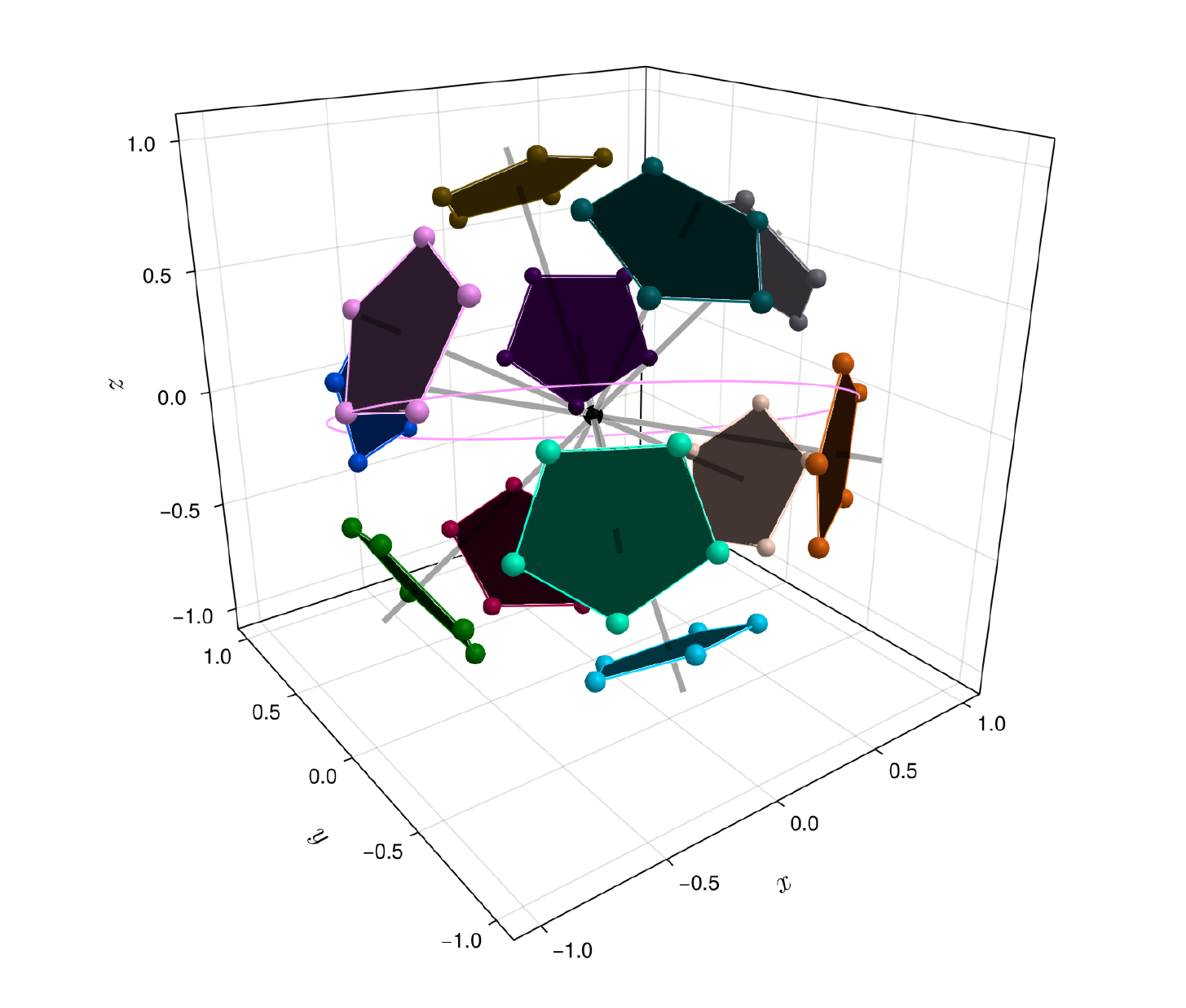}
        \caption{$\mathbb{I}$ \#2}
        \label{fig:cap-ico-circ2}
    \end{subfigure}
    \hfill
    \begin{subfigure}[b]{0.45\textwidth}
        \centering
        \includegraphics[width=\textwidth]{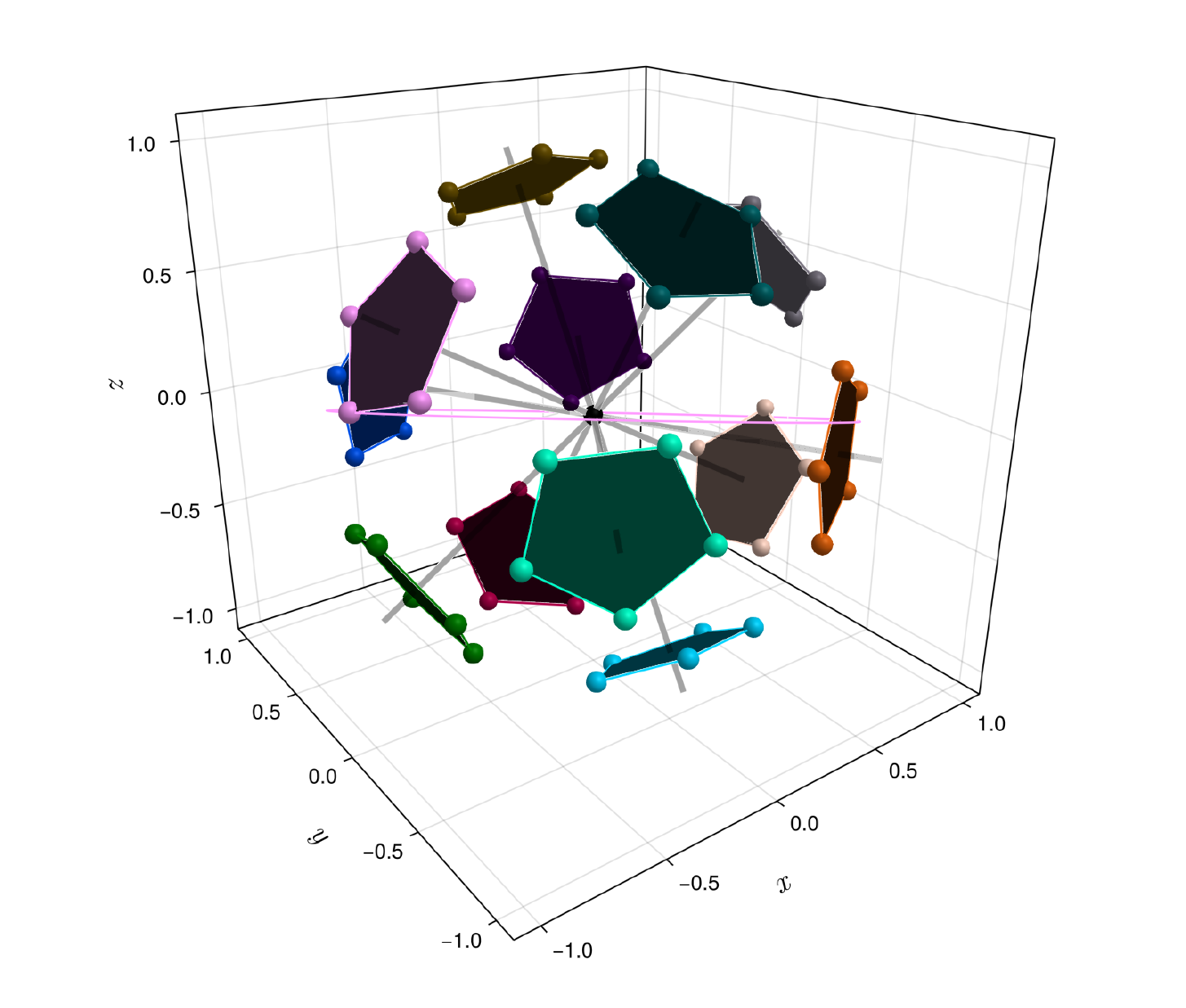}
        \caption{$\mathbb{I}$ \#3}
        \label{fig:cap-ico-circ3}
    \end{subfigure}

    \caption{The first six rows of Table~\ref{tab:cap-e0}, the
    circular solutions of limiting eccentricity $e_{0}=0$ as $\varepsilon\to0$.
    \rev{Each picture shows the $n+1$ bodies at $\tau = \pi$, corresponding to the apocentre of the eccentric orbits.} The central mass is the black ball at the origin, and the $n$ unit masses are one $H$-orbit, drawn as the vertices of a polyhedron whose lit faces give the perspective and display the symmetry. The reference body is the pink ball and its orbit the pink curve. The grey rods are rotation axes of $H$.}
    \label{fig:cap-grid1}
\end{figure}

\begin{figure}[htbp]
    \centering
    \begin{subfigure}[b]{0.45\textwidth}
        \centering
        \includegraphics[width=\textwidth]{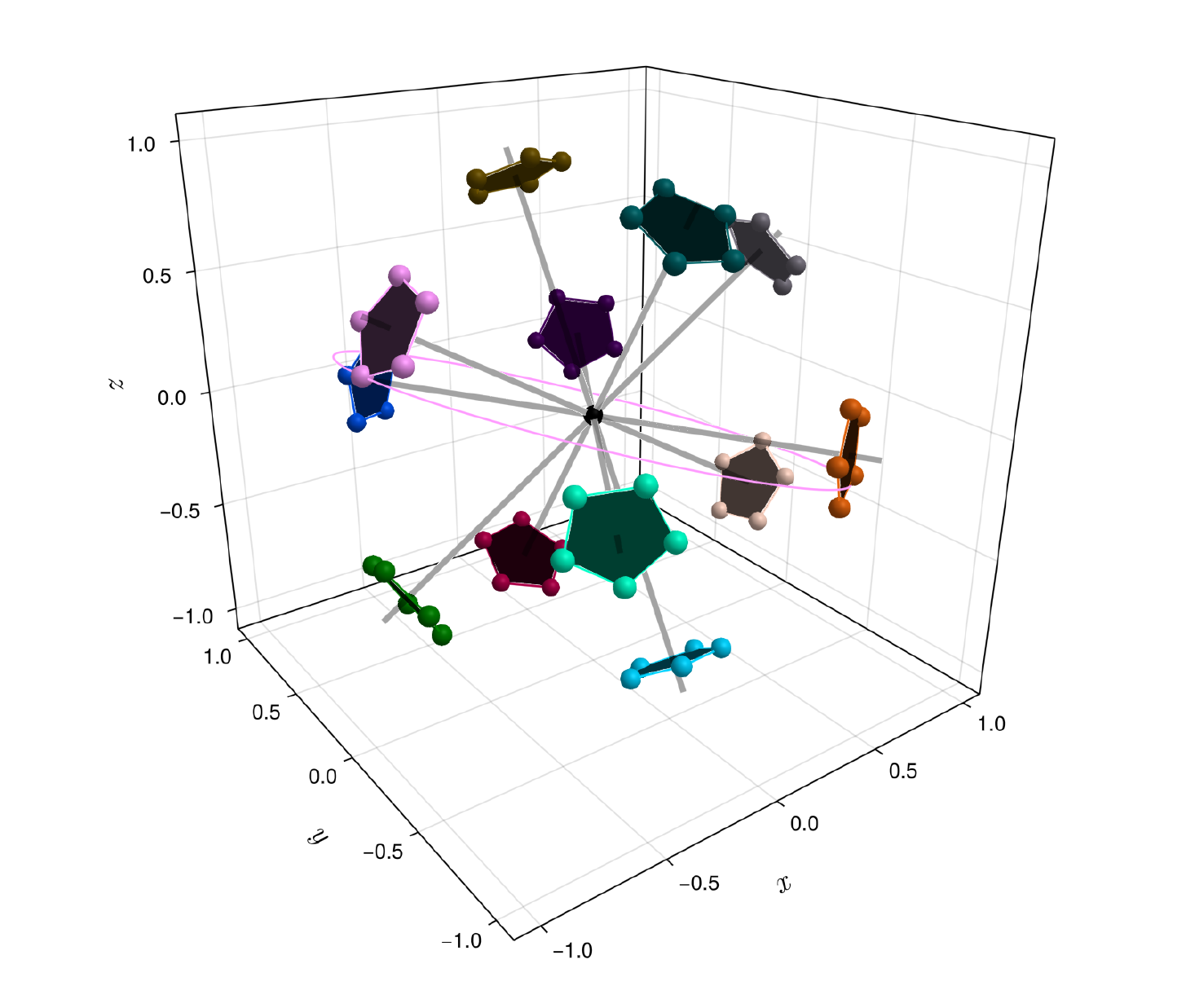}
        \caption{$\mathbb{I}$ \#4, $e_{0}=0$}
        \label{fig:cap-ico-circ4}
    \end{subfigure}
    \hfill
    \begin{subfigure}[b]{0.45\textwidth}
        \centering
        \includegraphics[width=\textwidth]{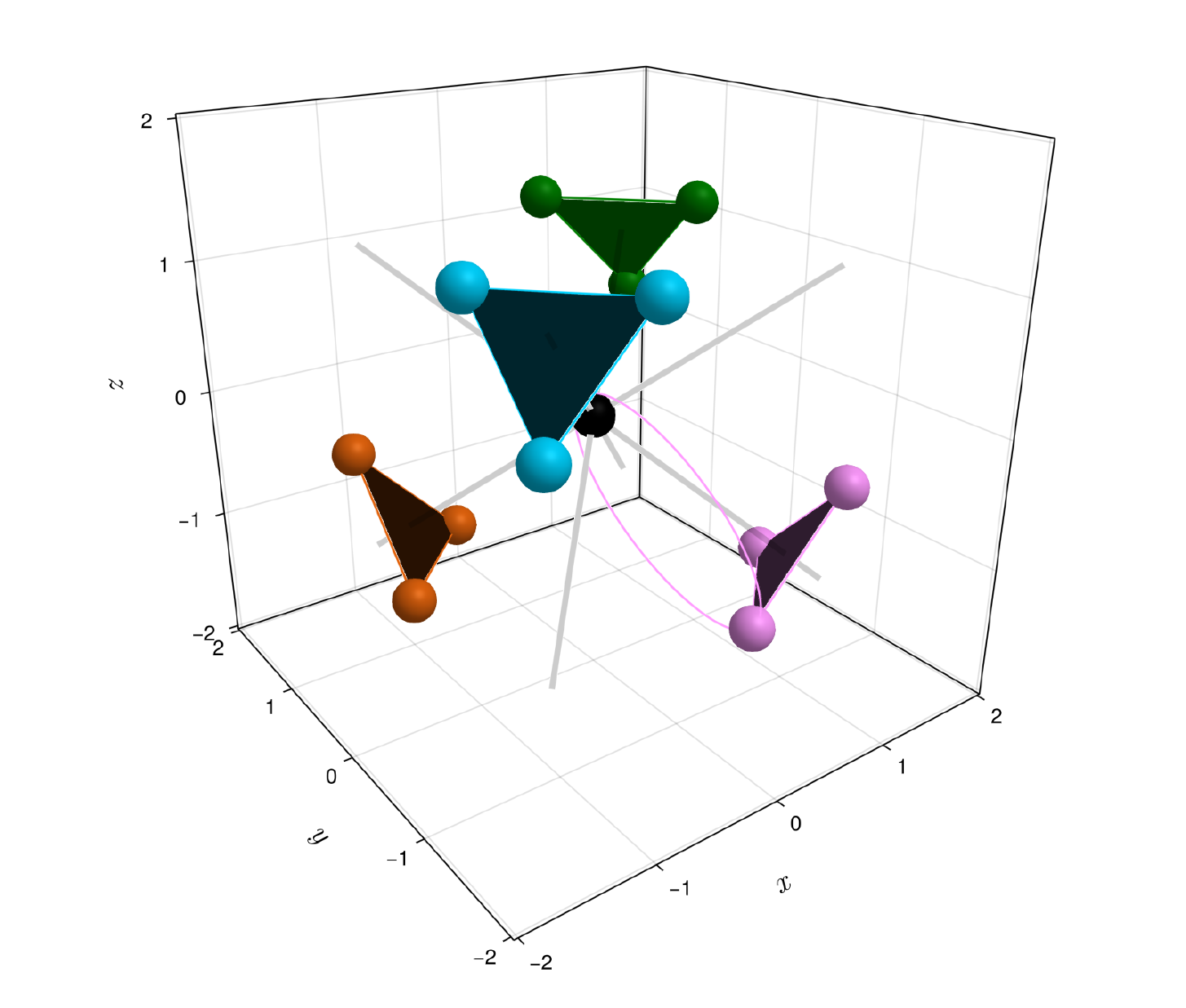}
        \caption{$\T$, $\bar e=0.844836$}
        \label{fig:cap-tet-ecc}
    \end{subfigure}
    
    \begin{subfigure}[b]{0.45\textwidth}
        \centering
        \includegraphics[width=\textwidth]{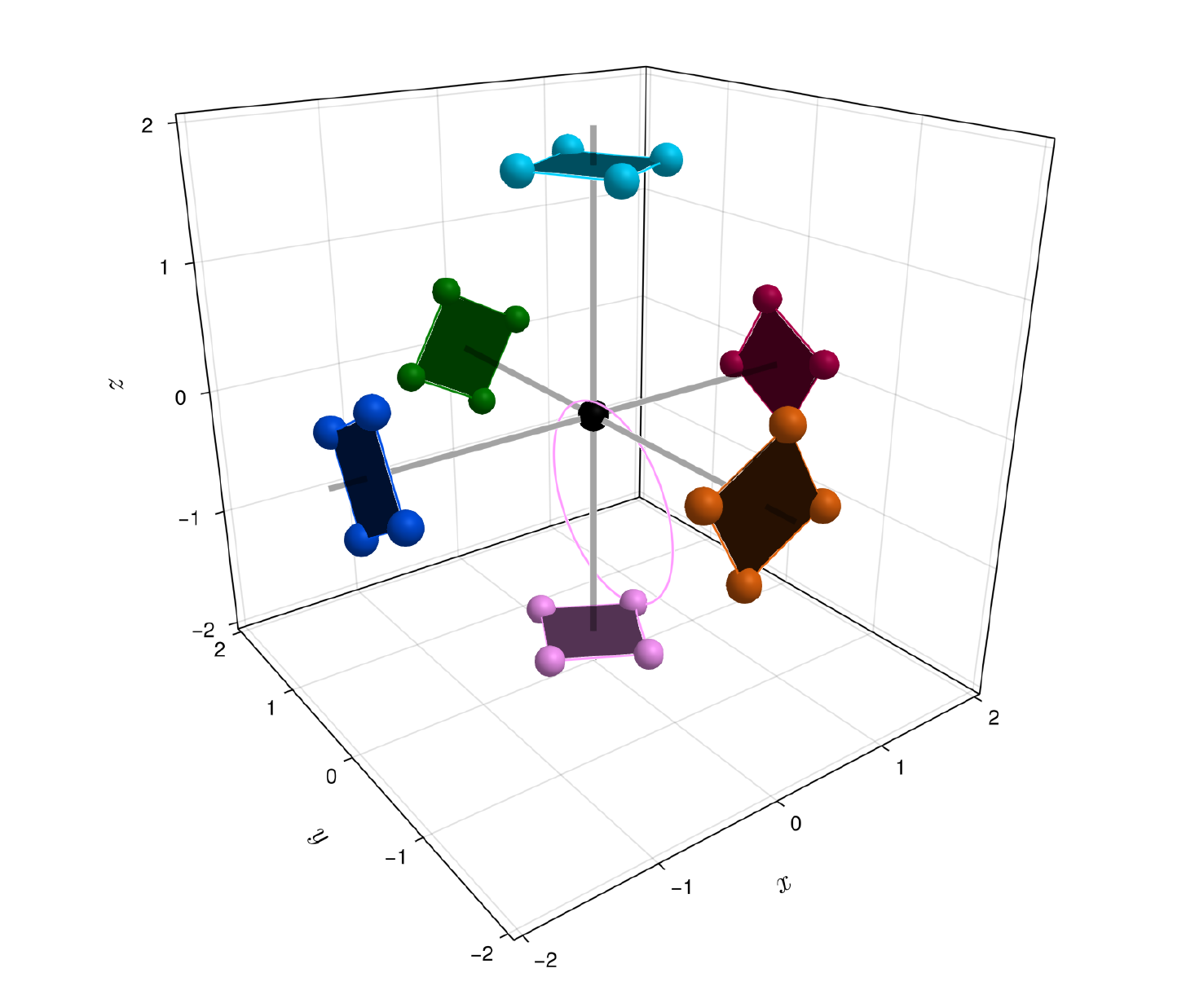}
        \caption{$\mathbb{O}$ \#1, $\bar e=0.873524$}
        \label{fig:cap-oct-ecc1}
    \end{subfigure}
    \hfill
    \begin{subfigure}[b]{0.45\textwidth}
        \centering
        \includegraphics[width=\textwidth]{O_ecc_2.pdf}
        \caption{$\mathbb{O}$ \#2, $\bar e=0.962477$}
        \label{fig:cap-oct-ecc2}
    \end{subfigure}

    \begin{subfigure}[b]{0.45\textwidth}
        \centering
        \includegraphics[width=\textwidth]{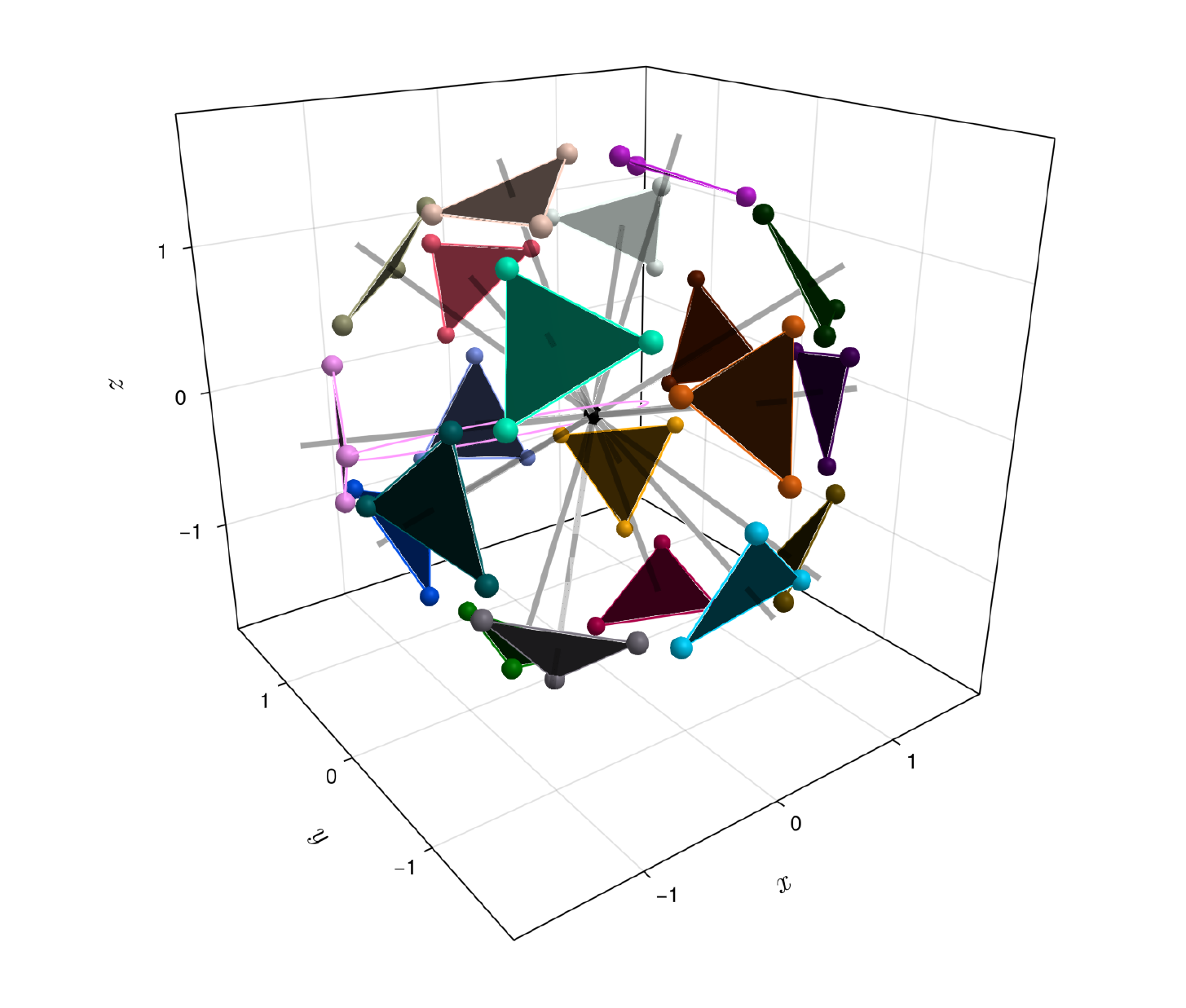}
        \caption{$\mathbb{I}$, $\bar e=0.679506$}
        \label{fig:cap-ico-ecc}
    \end{subfigure}

    \caption{The last row of Table~\ref{tab:cap-e0}, at
    $e_{0}=0$, and the four rows of Table~\ref{tab:cap-ecc}, at $e_{0}>0$. The
    drawing conventions are those of Figure~\ref{fig:cap-grid1}.}
    \label{fig:cap-grid2}
\end{figure}

\begin{remark}\label{rem:degenerate}
One critical orientation found by the search is absent from
Table~\ref{tab:cap-e0}, and it is absent for a reason of symmetry rather
than of numerics. It is the tetrahedral orientation at $e=0$ whose orbit-plane
normal is the body diagonal $(1,1,1)/\sqrt3$, that is
$(\psi_{1},\psi_{2})=(5.497787,\,5.667706)$. Its great circle is invariant under
the three-fold rotation about that diagonal, and the resulting extra symmetry
makes the eccentricity block of the Hessian \eqref{eq:H4} vanish identically. The
$4\times4$ Hessian is therefore singular there, Hypothesis~\ref{hyp:circular}
fails, and Theorem~\ref{thm:circular} does not apply. The orientation is
nonetheless a genuine critical point of the orientation equations, forced by the
three-fold symmetry. The degeneracy is of the reduced function and not of the
problem, and deciding whether a branch emanates from it requires an analysis
beyond the scope of this paper.
\end{remark}

\begin{remark}\label{rem:notexhaustive}
Tables~\ref{tab:cap-e0} and~\ref{tab:cap-ecc} are not claimed to list all the
critical points of $\Phi$. Each row is an existence
statement and nothing more, and the counts are lower bounds. Three regions of the
parameter space were not searched. The first is the set of eccentricities above
$e=0.9987$. The second is the set of orientations whose great circle passes closer
than $0.03$ to the collision set, which for $\mathbb{I}$ is a substantial part of
the domain, the icosahedral point of Theorem~\ref{thm:cap-ecc} lying at distance
$0.0317$ and so only just inside the searched region. The third consists of any
critical point contained, in the sweep over orbit planes, in a region smaller than
the grid spacing. Nothing in the paper depends on completeness, since a single row
of either table already yields a family of solutions.
\end{remark}

\begin{remark}[How the approximations were computed]\label{rem:numerics}
Three features of the numerical search produce the approximations $\bar p$ and
make them reliable, and we record them here since they play no part in the proofs.
First, at a fixed orbit plane the Kepler weight $W_{e}$ of \eqref{eq:W} is a
Poisson kernel in the variable $\lambda=(1-\sqrt{1-e^{2}})/e\in[0,1)$, so that
$\Phi$ is, \rev{on} that plane, a harmonic function of the eccentricity vector
$z=\lambda e^{i\psi_{3}}$ on the unit disc. Its interior critical points are
consequently always saddles, so that there is no interior maximum or minimum and
no descent method can reach them, and they are the \rev{zeros of an analytic function whose Taylor coefficients are the Fourier coefficients of the geometric sum $S$ of \eqref{eq:W} along the great circle, a polynomial once that series is truncated}. On the disc, therefore, nothing is iterated
and no root is missed. The natural radial variable is $\lambda$ and not $e$. Since
$\lambda=0.7$ already at $e=0.94$, a grid uniform in $e$ covers only half the disc
by area and misses the outer region, which is where the critical points lie.
Second, $\Phi$ is invariant under the normalizer $N_{O(3)}(H)$, of order
$48$ for $\T$ and $\mathbb{O}$, and $120$ for $\mathbb{I}$, so the search over
orbit planes may be confined to a spherical triangle of area
$4\pi/\abs{N_{O(3)}(H)}$, a fundamental domain of
its action on the sphere of orbit-plane normals.
Third, the integrand of \eqref{eq:Phi2} has a peak whose width is the distance
from the great circle to the nearest rotation axis of $H$, so the number of
quadrature nodes must be tied to that distance. At a fixed number of nodes the
Newton iteration converges onto quadrature noise near $e\to1$ and reports critical
points that do not exist.
\end{remark}

\bibliographystyle{siam}
\bibliography{references}

@article{MR4879804,
	author = {Lessard, Jean-Philippe and Pugliese, Alessandro},
	doi = {10.3934/dcdsb.2024181},
	fjournal = {Discrete and Continuous Dynamical Systems. Series B. A Journal Bridging Mathematics and Sciences},
	issn = {1531-3492,1553-524X},
	journal = {Discrete Contin. Dyn. Syst. Ser. B},
	mrclass = {37M20 (34C23 37Gxx 65P30)},
	mrnumber = {4879804},
	mrreviewer = {Joseph\ P\'aez Ch\'avez},
	number = {6},
	pages = {2135--2158},
	title = {Cusp bifurcations: numerical detection via two-parameter continuation and computer-assisted proofs of existence},
	url = {https://doi-org.proxy3.library.mcgill.ca/10.3934/dcdsb.2024181},
	volume = {30},
	year = {2025}}

@article{MR4930548,
	author = {van der Aalst, Lindsey and van den Berg, Jan Bouwe and Lessard, Jean-Philippe},
	doi = {10.1088/1361-6544/ade5e5},
	fjournal = {Nonlinearity},
	issn = {0951-7715,1361-6544},
	journal = {Nonlinearity},
	mrclass = {35C07},
	mrnumber = {4930548},
	number = {7},
	pages = {Paper No. 075029, 32},
	title = {Periodic localised traveling waves in the two-dimensional suspension bridge equation},
	url = {https://doi-org.proxy3.library.mcgill.ca/10.1088/1361-6544/ade5e5},
	volume = {38},
	year = {2025}}

@article{GaGN22,
	author = {Garc\'{\i}a-Azpeitia, Carlos and Garc\'{\i}a-Naranjo, Luis C.},
	doi = {10.1007/s00332-022-09792-y},
	journal = {Journal of Nonlinear Science},
	number = {3},
	pages = {39},
	title = {Platonic solids and symmetric solutions of the {$N$}-vortex problem on the sphere},
	volume = {32},
	year = {2022}}

@article{CDG16,
	author = {Calleja, Renato and Doedel, Eusebius and Garc\'{\i}a-Azpeitia, Carlos},
	doi = {10.1140/epjst/e2016-60009-y},
	journal = {The European Physical Journal Special Topics},
	number = {13-14},
	pages = {2741--2750},
	title = {Symmetry-breaking for a restricted {$n$}-body problem in the {M}axwell-ring configuration},
	volume = {225},
	year = {2016}}

@book{Dft_book,
	author = {Briggs, William L. and Henson, Van Emden},
	doi = {10.1137/1.9781611971514},
	eprint = {https://epubs.siam.org/doi/pdf/10.1137/1.9781611971514},
	publisher = {Society for Industrial and Applied Mathematics},
	title = {The DFT: An Owner's Manual for the Discrete Fourier Transform},
	url = {https://epubs.siam.org/doi/abs/10.1137/1.9781611971514},
	year = {1995}}

@misc{Constantineau_CAP_Code,
	author = {Constantineau, Kevin},
	note = {Available at: \url{https://github.com/BaronNashville/n_body_starting_point_code}},
	title = {{CAP Code}}}

@article{GaIz13,
	author = {C. Garc{\'\i}a-Azpeitia and J. Ize},
	doi = {https://doi.org/10.1016/j.jde.2012.08.022},
	issn = {0022-0396},
	journal = {Journal of Differential Equations},
	number = {5},
	pages = {2033-2075},
	title = {Global bifurcation of planar and spatial periodic solutions from the polygonal relative equilibria for the n-body problem},
	url = {https://www.sciencedirect.com/science/article/pii/S0022039612003361},
	volume = {254},
	year = {2013}}

@article{FGN11,
	author = {Fusco, G. and Gronchi, G. F. and Negrini, P.},
	doi = {10.1007/s00222-010-0306-3},
	issn = {1432-1297},
	journal = {Inventiones mathematicae},
	number = {2},
	pages = {283--332},
	title = {Platonic polyhedra, topological constraints and periodic solutions of the classical N-body problem},
	url = {https://doi.org/10.1007/s00222-010-0306-3},
	volume = {185},
	year = {2011}}

@article{BOZ19,
	author = {Boscaggin, Alberto and Ortega, Rafael and Zhao, Lei},
	doi = {10.1090/tran/7589},
	journal = {Transactions of the American Mathematical Society},
	pages = {677--703},
	title = {Periodic solutions and regularization of a {K}epler problem with time-dependent perturbation},
	url = {https://doi.org/10.1090/tran/7589},
	volume = {372},
	year = {2019}}

@book{MeHa91,
	author = {Meyer, Kenneth R. and Offin, Daniel C.},
	doi = {10.1007/978-3-319-53691-0},
	edition = {3},
	isbn = {978-3-319-53690-3},
	issn = {0066-5452},
	pages = {384},
	publisher = {Springer Cham},
	series = {Applied Mathematical Sciences},
	title = {Introduction to Hamiltonian Dynamical Systems and the N-Body Problem},
	url = {https://doi.org/10.1007/978-3-319-53691-0},
	year = {2017}}

@article{MR1420838,
	author = {Koch, Hans and Schenkel, Alain and Wittwer, Peter},
	coden = {SIREAD},
	doi = {10.1137/S0036144595284180},
	fjournal = {SIAM Review. A Publication of the Society for Industrial and Applied Mathematics},
	issn = {0036-1445},
	journal = {SIAM Rev.},
	mrclass = {39-01 (39B32 68N17)},
	mrnumber = {1420838 (97k:39001)},
	mrreviewer = {Kingsley Jones},
	number = {4},
	pages = {565--604},
	title = {Computer-assisted proofs in analysis and programming in logic: a case study},
	url = {http://dx.doi.org/10.1137/S0036144595284180},
	volume = {38},
	year = {1996}}

@book{MR2807595,
	author = {Tucker, Warwick},
	isbn = {978-0-691-14781-9},
	mrclass = {65-02 (65G30)},
	mrnumber = {2807595 (2012g:65004)},
	mrreviewer = {G. Alefeld},
	note = {A short introduction to rigorous computations},
	pages = {xii+138},
	publisher = {Princeton University Press, Princeton, NJ},
	title = {Validated numerics},
	year = {2011}}

@article{MR3444942,
	author = {van den Berg, Jan Bouwe and Lessard, Jean-Philippe},
	doi = {10.1090/noti1276},
	fjournal = {Notices of the American Mathematical Society},
	issn = {0002-9920,1088-9477},
	journal = {Notices Amer. Math. Soc.},
	mrclass = {68W30 (03B35 35K30 35K59 37M99)},
	mrnumber = {3444942},
	mrreviewer = {Hil\ G. E. Meijer},
	number = {9},
	pages = {1057--1061},
	title = {Rigorous numerics in dynamics},
	url = {https://doi.org/10.1090/noti1276},
	volume = {62},
	year = {2015}}

@book{MR3971222,
	author = {Nakao, Mitsuhiro T. and Plum, Michael and Watanabe, Yoshitaka},
	doi = {10.1007/978-981-13-7669-6},
	isbn = {978-981-13-7668-9; 978-981-13-7669-6},
	mrclass = {65G20 (35-04 65J15)},
	mrnumber = {3971222},
	pages = {xiii+467},
	publisher = {Springer, Singapore},
	series = {Springer Series in Computational Mathematics},
	title = {Numerical verification methods and computer-assisted proofs for partial differential equations},
	url = {https://doi-org.proxy3.library.mcgill.ca/10.1007/978-981-13-7669-6},
	volume = {53},
	year = {2019}}

@article{MR3990999,
	author = {G\'{o}mez-Serrano, Javier},
	doi = {10.1007/s40324-019-00186-x},
	fjournal = {SeMA Journal. Boletin de la Sociedad Espan\~{n}ola de Matem\'{a}tica Aplicada},
	issn = {2254-3902},
	journal = {SeMA J.},
	mrclass = {68T15 (35Q35 35R35 65G30)},
	mrnumber = {3990999},
	number = {3},
	pages = {459--484},
	title = {Computer-assisted proofs in {PDE}: a survey},
	url = {https://doi-org.proxy3.library.mcgill.ca/10.1007/s40324-019-00186-x},
	volume = {76},
	year = {2019}}

@article{CoGaLe2021,
	author = {Constantineau, Kevin and Garc{\'\i}a-Azpeitia, Carlos and Lessard, Jean-Philippe},
	date = {2021/10/28},
	doi = {10.1007/s12346-021-00532-3},
	id = {Constantineau2021},
	isbn = {1662-3592},
	journal = {Qualitative Theory of Dynamical Systems},
	number = {1},
	pages = {3},
	title = {Spatial relative equilibria and periodic solutions of the {C}oulomb {$(n+1)$}-body problem},
	url = {https://doi.org/10.1007/s12346-021-00532-3},
	volume = {21},
	year = {2022}}

@article{Yam98,
	author = {Yamamoto, Nobito},
	doi = {10.1137/S0036142996304498},
	journal = {SIAM Journal on Numerical Analysis},
	number = {5},
	pages = {2004--2013},
	title = {A numerical verification method for solutions of boundary value problems with local uniqueness by {B}anach's fixed-point theorem},
	volume = {35},
	year = {1998}}

@article{DLM07,
	author = {Day, Sarah and Lessard, Jean-Philippe and Mischaikow, Konstantin},
	doi = {10.1137/050645968},
	journal = {SIAM Journal on Numerical Analysis},
	number = {4},
	pages = {1398--1424},
	title = {Validated continuation for equilibria of {PDE}s},
	volume = {45},
	year = {2007}}

@article{Kra69,
	author = {Krawczyk, R.},
	doi = {10.1007/BF02234767},
	journal = {Computing},
	pages = {187--201},
	title = {{N}ewton-{A}lgorithmen zur {B}estimmung von {N}ullstellen mit {F}ehlerschranken},
	volume = {4},
	year = {1969}}

@article{Moo77,
	author = {Moore, Ramon E.},
	doi = {10.1137/0714040},
	journal = {SIAM Journal on Numerical Analysis},
	number = {4},
	pages = {611--615},
	title = {A test for existence of solutions to nonlinear systems},
	volume = {14},
	year = {1977}}

@article{Rum10,
	author = {Rump, Siegfried M.},
	doi = {10.1017/S096249291000005X},
	journal = {Acta Numerica},
	pages = {287--449},
	title = {Verification methods: rigorous results using floating-point arithmetic},
	volume = {19},
	year = {2010}}

@article{FeGr23,
	author = {Fenucci, Marco and Gronchi, Giovanni~F.},
	doi = {10.1007/s10884-021-10083-5},
	journal = {Journal of Dynamics and Differential Equations},
	pages = {1511--1559},
	title = {Symmetric constellations of satellites moving around a central body of large mass},
	volume = {35},
	year = {2023}}

@article{FeGr18,
	author = {Fenucci, Marco and Gronchi, Giovanni F.},
	journal = {Nonlinearity},
	number = {11},
	pages = {4935--4954},
	title = {On the stability of periodic {$N$}-body motions with the symmetry of {P}latonic polyhedra},
	volume = {31},
	year = {2018}}

@article{FuGr14,
	author = {Fusco, Giorgio and Gronchi, Giovanni F.},
	journal = {Journal of Dynamics and Differential Equations},
	pages = {817--841},
	title = {Platonic polyhedra, periodic orbits and chaotic motions in the {$N$}-body problem with non-{N}ewtonian forces},
	volume = {26},
	year = {2014}}

@article{Moser70,
	author = {Moser, J{\"u}rgen},
	journal = {Communications on Pure and Applied Mathematics},
	pages = {609--636},
	title = {Regularization of {K}epler's problem and the averaging method on a manifold},
	volume = {23},
	year = {1970}}

@article{FeTe04,
	author = {Ferrario, Davide L. and Terracini, Susanna},
	doi = {10.1007/s00222-003-0322-7},
	journal = {Inventiones mathematicae},
	number = {2},
	pages = {305--362},
	title = {On the existence of collisionless equivariant minimizers for the classical {$n$}-body problem},
	volume = {155},
	year = {2004}}

@article{Mar02,
	author = {Marchal, Christian},
	doi = {10.1023/A:1020128408706},
	journal = {Celestial Mechanics and Dynamical Astronomy},
	pages = {325--353},
	title = {How the method of minimization of action avoids singularities},
	volume = {83},
	year = {2002}}

@article{Bott54,
	author = {Bott, Raoul},
	doi = {10.2307/1969631},
	journal = {Annals of Mathematics (2)},
	pages = {248--261},
	title = {Nondegenerate critical manifolds},
	volume = {60},
	year = {1954}}

\end{document}